\documentclass[12pt,leqno]{amsart}
\usepackage{amssymb,amsmath,amsthm,bm}
\usepackage{graphicx}
\usepackage{color}
\usepackage[textsize=tiny]{todonotes}
\usepackage[normalem]{ulem}
\usepackage[bookmarksopen,bookmarksdepth=3,colorlinks,citecolor=red,pagebackref,hypertexnames=true]{hyperref}
\usepackage[msc-links,nobysame,non-sorted-cites, initials]{amsrefs}
\usepackage[inline]{enumitem}
\usepackage{mathtools}
\mathtoolsset{showonlyrefs}
\usetikzlibrary{quotes,arrows.meta}

\definecolor{darkblue}{RGB}{0,0,160}

\usepackage[lining]{libertine}   
\usepackage{cabin}
\usepackage[libertine]{newtxmath}
 
\usepackage[T1]{fontenc}

\def\e{{\rm e}}
\def\eps{\varepsilon}

\def\d{{\rm d}}
\def\dist{{\rm dist}}

\def\R {\mathbb{R}}

\def\Nn {\mathcal{N}}

\def\K {{\mathcal K}}

\def\Q {{\mathcal Q}}

\def\M {{\mathrm M}}

\def\1 {{\mbox{\boldmath 1}}}

\def \l {\langle}

\def \r {\rangle}

\def \and{\quad\text{and}\quad}

\def\ind{\cic{1}}
\newcommand{\cic}{\bm}
\newcommand{\sh}{\mathsf{sh}}
\newcommand{\mass}{\mathsf{mass}}
\newcommand{\pl}{\mathsf{af}}
\newcommand{\pr}{\Pi}
\def\Grdn {{\mathrm{Gr}(d,n)}}

 \def\Xint#1{\mathchoice
	   {\XXint\displaystyle\textstyle{#1}}%
	   {\XXint\textstyle\scriptstyle{#1}}%
	   {\XXint\scriptstyle\scriptscriptstyle{#1}}%
	   {\XXint\scriptscriptstyle\scriptscriptstyle{#1}}%
	   \!\int}
	 \def\XXint#1#2#3{{\setbox0=\hbox{$#1{#2#3}{\int}$}
	     \vcenter{\hbox{$#2#3$}}\kern-.5\wd0}}
	 
	 \def\avgint{\Xint-}

\def \no#1#2#3 {{\bf #1} (#3), #2.}
\def \eds#1#2#3 {#1, #2, #3.}

\newcounter{counter}

\numberwithin{equation}{section}
\numberwithin{counter2}{section}
\newtheorem{proposition}[subsection]{Proposition}
\newtheorem{theorem}[counter]{Theorem}

\newtheorem{corollary}{Corollary}
\newtheorem{lemma}[subsection]{Lemma} 
 
\newtheorem{conjecture}[counter]{Conjecture}
 
\theoremstyle{definition}
\newtheorem{definition}{Definition}
\newtheorem*{remark*}{Remark}
\newtheorem*{warn*}{A word of warning}

\newtheorem{remark}[subsection]{Remark} 
\theoremstyle{plain}
\numberwithin{corollary}{counter}

\numberwithin{figure}{section}

\makeatletter
\let\save@mathaccent\mathaccent
\newcommand*\if@single[3]{%
  \setbox0\hbox{${\mathaccent"0362{#1}}^H$}%
  \setbox2\hbox{${\mathaccent"0362{\kern0pt#1}}^H$}%
  \ifdim\ht0=\ht2 #3\else #2\fi
  }
\newcommand*\rel@kern[1]{\kern#1\dimexpr\macc@kerna}
\newcommand*\widebar[1]{\@ifnextchar^{{\wide@bar{#1}{0}}}{\wide@bar{#1}{1}}}
\newcommand*\wide@bar[2]{\if@single{#1}{\wide@bar@{#1}{#2}{1}}{\wide@bar@{#1}{#2}{2}}}
\newcommand*\wide@bar@[3]{%
  \begingroup
  \def\mathaccent##1##2{%
    \let\mathaccent\save@mathaccent
    \if#32 \let\macc@nucleus\first@char \fi
    \setbox\z@\hbox{$\macc@style{\macc@nucleus}_{}$}%
    \setbox\tw@\hbox{$\macc@style{\macc@nucleus}{}_{}$}%
    \dimen@\wd\tw@
    \advance\dimen@-\wd\z@
    \divide\dimen@ 3
    \@tempdima\wd\tw@
    \advance\@tempdima-\scriptspace
    \divide\@tempdima 10
    \advance\dimen@-\@tempdima
    \ifdim\dimen@>\z@ \dimen@0pt\fi
    \rel@kern{0.6}\kern-\dimen@
    \if#31
      \overline{\rel@kern{-0.6}\kern\dimen@\macc@nucleus\rel@kern{0.4}\kern\dimen@}%
      \advance\dimen@0.4\dimexpr\macc@kerna
      \let\final@kern#2%
      \ifdim\dimen@<\z@ \let\final@kern1\fi
      \if\final@kern1 \kern-\dimen@\fi
    \else
      \overline{\rel@kern{-0.6}\kern\dimen@#1}%
    \fi
  }%
  \macc@depth\@ne
  \let\math@bgroup\@empty \let\math@egroup\macc@set@skewchar
  \mathsurround\z@ \frozen@everymath{\mathgroup\macc@group\relax}%
  \macc@set@skewchar\relax
  \let\mathaccentV\macc@nested@a
  \if#31
    \macc@nested@a\relax111{#1}%
  \else
    \def\gobble@till@marker##1\endmarker{}%
    \futurelet\first@char\gobble@till@marker#1\endmarker
    \ifcat\noexpand\first@char A\else
      \def\first@char{}%
    \fi
    \macc@nested@a\relax111{\first@char}%
  \fi
  \endgroup
}
\makeatother

\begin{document}

\title[Maximal subspace averages]{Maximal subspace averages} 

\author[F. Di Plinio]{Francesco Di Plinio} \address{\noindent Department of Mathematics and Statistics, Washington University in St. Louis}
  \email{\href{mailto:francesco.diplinio@wustl.edu}{\textnormal{francesco.diplinio@wustl.edu}}}
\thanks{F. Di Plinio is partially supported by the National Science Foundation under the grants   NSF-DMS-2000510, NSF-DMS-2054863}

\author[I. Parissis]{Ioannis Parissis}
\address{Departamento de Matem\'aticas, Universidad del Pais Vasco, Aptdo. 644, 48080 Bilbao, Spain and Ikerbasque, Basque Foundation for Science, Bilbao, Spain}

\email{\href{mailto:ioannis.parissis@ehu.es}{\textnormal{ioannis.parissis@ehu.es}}}
\thanks{I. Parissis is partially supported by the project PGC2018-094528-B-I00 (AEI/FEDER, UE) with acronym ``IHAIP'', grant T1247-19 of the Basque Government and IKERBASQUE.}


\begin{abstract} We study  maximal  operators
associated to  singular averages along finite subsets $\Sigma$ of the  Grassmannian $\Grdn$ of $d$-dimensional subspaces of $\R^n$. 
The  well studied $d=1$ case corresponds to the the directional maximal function with respect to arbitrary finite subsets of $\mathrm{Gr}(1,n)=\mathbb S^{n-1}$.  
We provide a systematic study of all cases $1\leq d<n$ and  prove essentially  sharp $L^2(\R^n)$ bounds for the maximal subspace averaging operator in terms of the cardinality of $\Sigma$, with no assumption on the structure of $\Sigma$. In the   codimension $1$ case, that is $n=d+1$, we prove the precise critical weak $(2,2)$-bound.

Drawing on the analogy between   maximal subspace averages and $(d,n)$-Nikodym maximal averages, we also    formulate the appropriate maximal Nikodym conjecture for general $1<d<n$ by providing examples that determine the critical $L^p$-space for the $(d,n)$-Nikodym problem. Unlike the $d=1$ case, the maximal  Kakeya and Nikodym problems are shown not to be equivalent when $d>1$. In this context, we prove the best possible $L^2(\R^n)$-bound for the $(d,n)$-Nikodym maximal function for all combinations of dimension and codimension. 

Our estimates rely on Fourier analytic almost orthogonality principles, combined with polynomial partitioning, but we also use spatial analysis based on the precise calculation of intersections of $d$-dimensional plates in $\R^n$.
\end{abstract}
\maketitle

\section{Introduction}
For $\sigma \in \mathrm{Gr}(d,n)$, the Grassmannian of $d$-dimensional subspaces of $\R^n$, the scale $s$ subspace average of $f\in \mathcal C(\R^n)$ is 
\[
\langle f \rangle_{s,\sigma} (x) \coloneqq \int\displaylimits_{sB_n\cap \sigma} f(x-y) \frac{\d y}{s^d}, \qquad x\in \R^n.
\]
where $B_n\subset \R^n$ is the $n$-dimensional unit ball  centered at the origin, $sB_n$ is its concentric dilate, and $\d y=\d\mathcal L^d(y)$ denotes the Lebesgue measure on $\sigma\in\Grdn$. Fubini's theorem ensures that, up to the dimensional constant $\mathcal L^d(B_d)$, the map $f\mapsto \langle f \rangle_{s,\sigma} $ preserves the $L^1(\R^n)$-norm of $f$ and contracts all $L^p(\R^n)$-norms for $1<p\leq\infty$. 

The general concern of this article is the $L^p$-behavior of the positive maps
\[
f\mapsto \langle f \rangle_{s(\cdot),\sigma(\cdot)} (\cdot)
\] 
corresponding to a measurable choice of $\sigma \in \mathrm{Gr}(d,n)$ and scale $s>0$ depending on the point $x\in \R^n$.  When $d=n$,  these maps are pointwise controlled by the standard Hardy-Littlewood maximal operator for any pair of choice functions $\sigma,s$.  The singular cases $d<n$ give rise to a family of nontrivial problems of intrinsic relevance within the theory of differentiation of integrals, and possessing applications to  singular and oscillatory operators,   geometric measure theory and partial differential equations.

A central example is the classical question, attributed to Zygmund, of characterizing the class of  planar vector fields that differentiate $L^2(\R^2)$ functions. The corresponding singular integral variant of Zygmund's question, usually attributed to Stein, asks whether Lipschitz vector-fields of directions allow for a weak $(2,2)$ bound for the corresponding directional Hilbert transform, after suitable truncation. 

The case of choice functions $\sigma(\cdot)$ whose range is a \emph{finite} subset $\Sigma\subset \mathrm{Gr}(d,n)$ is also of particular importance. The averaging operator is of maximal nature, that is for $s\in S\subset (0,\infty)$
\[
 \M_{\Sigma,\{s\}}f\coloneqq \sup_{\sigma\in\Sigma}  \langle |f| \rangle_{s,\sigma} (\cdot),\qquad \M_{\Sigma,S}f\coloneqq \sup_{s\in S} \M_{\Sigma,\{s\}}f.
\]

The study of one-dimensional directional averages in $\R^n$, corresponding to $\mathrm{Gr}(1,n)$, is   connected to the problem of determining the Hausdorff dimension of Kakeya sets in $\R^n$. There is a classification of such questions for all $d<n$, with $\mathrm{Gr}(d,n)$ corresponding to the problem of studying the existence of $(d,n)$ Besicovitch sets. For the study of directional averages as above in relation to the Kakeya-Besicovitch category of problems, it suffices to consider finite subsets $\Sigma$ which are uniformly distributed at some distinct scale $\delta$; more specifically, the range $\Sigma$ of $\sigma(\cdot)$ is a maximal $\delta$-net in $\mathrm{Gr}(d,n)$. This in turn leads to  seeking for  $L^p$-bounds for the corresponding maximal directional averages as a function of $\delta$.

\subsection{Thin $(d,n)$-averages for arbitrary finite $\Sigma\subset \Grdn$} The case of of more general maximal directional averages, where $\sigma$ takes values in a finite but \emph{arbitrary} subset $\Sigma\subset\mathrm{Gr}(d,n)$, is in general much harder as there is no distinct scale in the set of directions, and any suitable method must  make up for the lack of uniform density in $\Sigma$. This obstruction is already present in the \emph{single scale problem}. Our first main result is an essentially complete description of the sharp $L^2(\R^n)$-bounds in the whole range of dimensions and codimensions $1\leq d < n$.

\begin{theorem}\label{t:sharpss}
For all $N>0$, $1\leq d < n $ there holds
\begin{equation}
\label{e:tAintro}
\sup_{s>0}\sup_{\substack{\Sigma\subset \mathrm{Gr}(d,n)\\  \#\Sigma \leq N} } \left\|\mathrm{M}_{\Sigma,\{s\}} \right\|_{L^2(\R^n)} \lesssim N^{\frac{(n-d-1)}{2(n-d)}} (\log N)^{\frac{n-d}{2}}
\end{equation}
with an implicit numerical constant depending only upon $d,n$. This bound is  sharp in terms of $N$ when $n=d+1$ and sharp up to the logarithmic factor in general.
\end{theorem}

  For the special case $d=1$ and $n=3$,  we obtain an improved version of Theorem \ref{t:sharpss} where the corrective logarithmic term appearing in the right hand side of \eqref{e:tAintro} is replaced  by an arbitrarily iterated logarithm of $N$. The precise statement is given in Theorem~\ref{t:sharp3d} of Section~\ref{s:5}.

Before proceeding with the description of our second group of results, some remarks are in order. The case $d=1<n=2$ is due to Katz, \cite{KB}. It should be noted that two-dimensional versions of the theorem above are somewhat related to the   resolution of the Kakeya conjecture in two dimensions and to earlier results of C\'ordoba \cites{corpoly,Cor77} and Str\"omberg \cite{Stromberg}. We note here that Theorem~\ref{t:sharpss} recovers the sharp result of Katz, \cite{KB}, for single scale maximal directional averages on $L^2(\R^2)$, so necessarily $d=1$, in fact with a new and independent proof. Note that Theorem~\ref{t:sharpss} is a single-scale but with $\Sigma\subset \mathrm{Gr}(d,n) $ \emph{arbitrary}.
The  $d=1,n=3$ improved result of Theorem \ref{t:sharp3d} is itself an amelioration of \cite[Theorem B]{DPPalg}. The proof technique for Theorem \ref{t:sharp3d} combines the polynomial partitioning ideas of \cite{DPPalg} with a new algebraic almost-orthogonality result, Theorem \ref{t:alldortho}, which  is of independent interest and may be seen as a higher-dimensional analogue of the well-known Alfonseca-Soria-Vargas almost-orthogonality principle \cite{ASV}.

\subsection{Sharp critical estimates for multiscale $(d,d+1)$-averages}
In the cases of maximal directional averages in arbitrary dimension and codimension $1=n-d$, we give a final theorem that provides the best possible $L^p$-bounds, $p\geq 2$, for the corresponding maximal averages with respect to arbitrary finite $\Sigma\subset \mathrm{Gr}(d,d+1)$. It should be noted that in this case we can actually tackle the multiscale problem in a sharp fashion. An important feature shared by the codimension $1$ problems is that the critical integrability space for the maximal function $\mathrm M_{\Sigma,S}$ is $L^2(\R^n)$. 

\begin{theorem}\label{thm:codim1_intro} Let $n=d+1$ and $\Sigma\subset \mathrm{Gr(d,d+1)}$ be a finite set. Then
	\[
	\| \M_{\Sigma,(0,\infty)}:L^2(\R^n)\to L^{2,\infty}(\R^n)\|\lesssim (\log \#\Sigma)^{\frac 12}
	\]
and
	\[
	\| \M_{\Sigma,(0,\infty)}:L^p(\R^n)\to L^{p}(\R^n)\|\lesssim 
	 \begin{cases} 
		 (\log \#\Sigma)^{\frac{1}{p}},&\qquad p>2,\vspace{.4em}
		 \\
    	\log \#\Sigma,&\qquad p=2.
	\end{cases}
	\]
These bounds are best possible in terms of the dependence on the cardinality $\#\Sigma$.
\end{theorem}
  For the sharpness of the obtained bounds, we send to Proposition~\ref{prop:cex}. The main tool in the proof of the upper bounds is a directional Carleson embedding theorem for suitable Carleson sequences indexed by $\delta$-plates, see Theorem \ref{thm:carleson} in Section \ref{s:cod1}. More specifically, Theorem~\ref{thm:codim1_intro} is obtained by an application of Theorem \ref{thm:carleson}  to the adjoint of the (linearized) maximal operator $\mathrm M_{\Sigma,S}$. More general applications of the directional Carleson theorem are revealed by couplings with time-frequency  analysis. Along these lines, as a representative sample of the scope of   Theorem \ref{thm:carleson}, we derive from it a sharp   estimate for the Rubio de Francia square function associated to $N$ well-distributed conical sectors in $\R^n$, see Theorem~\ref{t:sfe}.

\subsection{$L^2$-estimates for the $(d,n)$-Nikodym maximal operator}  With precise definitions and discussion to come later, we define the $(d,n)$-Nikodym maximal function on $\R^n$ by\[
\mathcal N_\delta f (x) \coloneqq \sup_{\sigma \in\Grdn}\fint_{x+T_\delta(\sigma)}|f(y)|\,\d y,\qquad x\in\R^n,
\]
where $T_\delta(\sigma)$ is a $\delta$-neighborhood of $\sigma\cap B_n(1)$; see \eqref{e:plate}.  This is the maximal operator coupled with the dimensional analysis of $(d,n)$-Nikodym sets for general $1\leq d<n$.
 As   discussed in \S\ref{s:kakeya}, suitable $L^p(\R^n)$-bounds for $N_\delta$ imply corresponding lower bounds for the Hausdorff dimension of $(d,n)$-Nikodym sets. Our main result here is the following

\begin{theorem}\label{thm:L2nik} Let $1\leq d<n$ and $\delta>0$. There holds
	\[
	\| \Nn_\delta  :\, L^{2}(\R^n )\to L^{2,\infty}(\R^n) \| \lesssim 
	\begin{cases}
		\delta^{-\frac{n-d-1}{2}},&\qquad 1\leq d<n-1,
		\\
		\sqrt{\log (\delta^{-1})},&\qquad d=n-1,
	\end{cases}
	\]
	and 
	\[
	\| \Nn_\delta  :\, L^{2}(\R^n )\to L^{2}(\R^n) \| \lesssim 
	\begin{cases}
		\delta^{-\frac{n-d-1}{2}}\sqrt{\log\delta^{-1}},&\qquad 1\leq d<n-1,
		\\
		\log (\delta^{-1}),&\qquad d=n-1,
	\end{cases}
	\]
The weak $(2,2)$ bound is best possible while the strong $(2,2)$ bound is best possible up to the logarithmic factor $\sqrt{\log(\delta^{-1})}$.
\end{theorem}

\subsection{Motivation and Background} The study of maximal directional averages has a long history,   motivated for instance by the Kakeya,   Zygmund  and Stein conjectures mentioned above. 
Maximal  averages   along subsets of $\mathrm{Gr}(d,n)$ are also quantitatively connected with the behavior of the ball and polygon multipliers and the convergence of Fourier series in higher dimensions, as well as to  square functions formed by frequency projections to the corresponding dual subsets of the frequency domain. This last theme is explored in detail in \cite{APHPR} in dimension $n=2$. As mentioned above, we present  an application of this type, for arbitrary dimension $n$, in \S\ref{sec:sf} below.

The study of $\mathrm{Gr}(d,n)$-averages can be classified according to the structure of   $\Sigma\subset \mathrm{Gr}(d,n)$. This classification is more established in the case $d=1$, where  $\Sigma\subset \mathbb S^{n-1}$ and three particular  cases of interest arise.

First, lacunary sets  \cites{CorFeflac,NSW,SS,Strlac} are the only  {infinite} subsets of $\mathbb S^{n-1}$ that give rise to bounded directional maximal functions. This characterization is due to Bateman in $n=2$, and albeit in a weaker form, to Parcet and Rogers \cite{ParRog}, and relies ultimately on the existence of Kakeya and Nikodym sets. Second,  sharp $L^p(\R^n)$-bounds for maximal averages along $\delta$-uniform sets of directions $\Sigma_\delta\subset \mathbb S^{n-1}$ are the subject of the maximal Nikodym conjecture, whose lower bounds tells us \emph{a fortiori} that $\|\M_{\Sigma_\delta,\{1\}}\|_{L^p(\R^n)}$ cannot be independent of $\delta$ for any $p\in(1,\infty)$.
In fact, the maximal Nikodym conjecture is formally weaker than statements involving $\M_{\Sigma_\delta,\{1\}}$, dealing with averages along $\delta$-tubes instead of \emph{thin}, $\delta$-uniformly spaced  averages of the form $\M_{\Sigma_\delta,\{1\}}$.

In two dimensions the sharp bounds for $\M_{\Sigma_\delta,S}$ are known, see for example \cite{corpoly,Cor77,Stromberg}.  In higher dimensions $n>3$ the Kakeya conjecture is open but several partial and very significant results are available; see for example \cite{HRZ} and the references therein. Best possible $L^2(\R^n)$ bounds for multiscale averages along uniformly distributed sets in $\mathbb S^{n-1}$ recently appeared in \cite{Kim}. 

The problem of sharp $L^p$-bounds for  $\M_{\Sigma,S}$  when   $\Sigma\subset \mathbb S^{n-1}$ is instead arbitrary is also fully solved when $n=2$ \cite{Katz,KB}. However, in particular when $n>2$, it is in general much harder than the $\delta$-net case,  as there is no distinct scale on the set of directions, or  alternatively, there is no fixed density of directions. This lack of structure   does not allow for e.g.\ the use of induction on scales, and new tools are necessary.   Recently, these authors  proved in \cite{DPPalg} essentially sharp $L^2(\R^n)$-bounds for $\M_{\Sigma,\{1\}}$ for arbitrary finite $\Sigma\subset \mathbb S^{n-1}$ or $\Sigma \subset Z$ with $Z$ algebraic submanifold of $\mathbb S^{n-1}$, via the  polynomial method.

Moving away from the restriction $d=1$, we introduce in this paper a family of problems related to averages with respect to $\mathrm{Gr}(d,n)$ for general $1\leq d <n$. Such Radon-type transforms have been studied in several forms in the literature, for example for the $(d,n)$-Kakeya maximal function
\[
\mathcal K_\delta f(L)\coloneqq  \sup_{x\in\R^n}\int_{L+x} f(y)\,\d y,\qquad L\in \mathrm{Gr}(d,n),
\]
where $f$ is a suitable function in $\R^n$. In analogy to the case $d=1$, suitable $L^p(\R^n)\to L^q(\mathrm{Gr}(d,n))$ bounds for the $(d,n)$-Kakeya operator, or for corresponding $d$-plane transforms, relate to the existence and dimension of $(d,n)$-Besicovitch sets; see \cite{Mattila}*{\S24}, \cites{Oberlin,Oberlin_thesis}, and the discussion in Section~\ref{s:kakeya}. Our point of view is different, unifying the study of $d$-plane averages for all $1\leq d<n$ in the form of \emph{thin} subspace averages $\mathrm{M}_{\Sigma,\{s\}}$, possibly at different scales $s$, as operators acting on $L^p(\R^n)$. As in the case $d=1$ the structure of the subset $\Sigma\subset \mathrm{Gr}(d,n)$ under consideration is of paramount importance. In this light, our setup is new and, for example, there is currently no definition of lacunary subsets of $ \mathrm{Gr}(d,n)$. In this paper we thoroughly study the cases of arbitrary $\Sigma\subset  \mathrm{Gr}(d,n)$. The case of uniformly distributed subsets  $\Sigma_\delta\subset  \mathrm{Gr}(d,n)$ is also important and is implicit in the study of the $(d,n)$-Kakeya and Nikodym maximal operators in Section~\ref{s:kakeya} and more precisely in the statements of Proposition~\ref{prop:kakeya} and Theorem~\ref{thm:L2nik}. There is again a critical integrability space $L^{p(d,n)}(\R^n)$ for $\M_{\Sigma_\delta,\{1\}}$ relating to the problem of existence and dimension of $(d,n)$-Nikodym sets. 

An important difference is that for $d>1$ the $(d,n)$-Kakeya conjecture and the $(d,n)$-Nikodym conjecture appear to be independent of each other and so are the corresponding critical exponents $p(d,n)$; note that this in stark with the case $d=1$ where the maximal Kakeya and maximal Nikodym conjectures are equivalent; see \cite{Tao}. For example a well known result of Falconer, \cite{Falc}, implies the there are no $(d,n)$ Besicovitch sets for $d>n/2$ while another result of Falconer, \cite{FalcPLMS}, shows that there exist $(d,n)$-Nikodym sets for all $1\leq d<n$. This difference is also reflected to the fact that, unlike the case $d=1$, the possibility of $\delta$-free bounds for the $(d,n)$-Kakeya operator $\mathcal K_\delta$ is not excluded. An instance of this is contained in the statement of Proposition~\ref{prop:kakeya}. On the other hand this is not the case for the  $(d,n)$-Nikodym maximal operator, nor for $\M_{\Sigma,\{s\}}$, as exhibited in Theorem~\ref{thm:L2nik}, in accordance to the previously mentioned result of Falconer. A more general study of $d$-plane averages of the form $\M_{\Sigma,S}$ and of $(d,n)$-Nikodym maximal operators for $1<d<n$ is motivated by these connections and nuances. We introduce the maximal multiscale problem with $S=(0,\infty)$ and generic $\Sigma\subset \mathrm {Gr}(d,n)$ and manage to fully resolve the problem in the codimension $1=n-d$ case, in the form of Theorem~\ref{thm:codim1_intro} above. At the same time we formulate a general $(d,n)$-Nikodym conjecture and discuss the critical integrability index for general $d$. From that point of view Theorem~\ref{thm:L2nik} stated above describes a sharp but subcritical estimate for the $(d,n)$-Nikodym function, and corresponding maximal conjecture.

The investigations in the current paper lead to several natural questions concerning the $L^p(\R^n)$-bounds for $d$-subspace averages $\M_{\Sigma,S}$, and especially the study of such $L^p$-norms close to the critical exponent $p=p(d,n)$ is particularly interesting, and consistently hard. The current paper addresses in particular all the $L^2(\R^n)$-bounds for such operators in a sharp fashion, whether $L^2(\R^n)$ happens to be subcritical as in the case $n>2$, $d<n-1$, or critical as in the case $d=n-1$ in any dimension.

\subsection{Methodology} In this paper we employ a mix of geometric, Fourier analytic, and polynomial methods. The latter technique for the study of directional maximal operators along by arbitrary sets of directions was introduced in \cite{DPPalg}. Using a polynomial partition we divide the set of directions into subsets (cells) of controlled cardinality and such the boundary of these subsets is an algebraic variety of controlled degree. The properties of this partition allows us to prove an almost orthogonality principle via Fourier methods, as the algebraic nature of the boundary of the cells (wall) yields suitable overlap estimates for the relevant Fourier multipliers. Using this scheme we prove a general almost orthogonality principle for single scale directional averages defined with respect to an arbitrary set of directions. In the current paper we apply this principle to yield essentially sharp bounds for maximal $d=1$-dimensional averages given by arbitrary directions on the sphere. The application of the polynomial partitioning scheme to the case of $d$-dimensional averages in $\R^n$ will require a suitable polynomial partition on the Grassmannian $\Grdn$ and will be taken up in a future work. However, for general codimension $n-d$ we present an alternative argument that recovers almost sharp $L^2$-bounds for $d$-dimensional single-scale averages given by \emph{arbitrary} subsets $\Sigma\subset \Grdn$. Indeed this approach misses the conjectured sharp $L^2$-bound, which is polynomial in $\#\Sigma$, by a logarithmic factor in $\#\Sigma$.

In two cases, we employ a different point of view in order to sharp $L^2(\R^n)$-bounds via space analysis. This is particular efficient when proving $L^2(\R^n)$-bounds for maximal directional $d$-plane averages with $n=d+1$. It is important to note that for this codimension-1 case, the space $L^2(\R^n)$ is critical and we do prove the best possible bound in Theorem~\ref{thm:codim1_intro}, in fact even for the multiscale maximal function. The approach, inspired by the works  of Katz \cite{Katz} and Bateman   \cite{Bat1v},  is via a $TT^*$-argument on the adjoint of the linearized maximal operator and an appeal to a suitable directional Carleson embedding theorem. 
In \cite{APHPR} this method was elaborated into a directional Carleson embedding theorem for suitable directional Carleson sequences. Here we suitably adapt  the geometric part of the argument, resulting in a corresponding directional Carleson embedding for sequences indexed by $(n-1)$-dimensional plates in $\R^n$ and satisfying a Carleson condition adjusted to the geometry of such plates.

A second $TT^*$ instance   appears in the proof of sharp $L^2(\R^n)$-bounds for the  $\Grdn$-Nikodym maximal operator.
 Here we are able to exploit specific structure of the nets $\Sigma_\delta$ and prove explicit estimates for the volume of the intersections of such $\delta$-plates in all combinations of dimension and codimension. These volume estimates and the $TT^*$ argument yield the sharp bound for the $(d,n)$-Nikodym maximal operator of Theorem~\ref{thm:L2nik}.

\subsection{Notation} The purpose of this  paragraph is to provide easy reference for a few central definitions, in particular for the several maximal and averaging operators that appear throughout the paper. 

\begin{itemize} 
\item The notation  $B_{k}(z)$ is reserved for the unit ball in $\R^{k}$ centered at $z\in\R^k$, and we write $B_k\coloneqq B_k(0)$. 
\item If $\sigma $ is a  subspace of $\R^n$, we denote by $\pr_\sigma$ the corresponding orthogonal projection. For $v\in\mathbb S^{n-1}$ we abuse notation and write $\pr_v$ instead of $\pr_{\mathrm{span}\{v\}}$.
\item The notation $\mathrm M$ is reserved for the Hardy-Littlewood maximal operator on the corresponding  $\R^n$.
\item For $\Sigma\subset \Grdn$ and $S\subset (0,\infty)$,  $\mathrm M_{\Sigma,S}$ stands for the maximal averaging operator with respect to \emph{thin} plates $\sigma \cap B_n(s)$ with $\sigma\in \Sigma$ and choice of scale $s\in S$. In symbols,
\begin{equation}\label{e:geommax}
\mathrm{M}_{\Sigma,S} f(x)\coloneqq \sup_{\substack{s\in S\\\sigma \in\Sigma}} \langle |f| \rangle_{\sigma,s}(x), \qquad x\in \R^n.
\end{equation}
The case $S=\{s\}$ for some $s>0$ is the single-scale case and will appear in several places below. When $S=\{1\}$ we simplify the notation to $\mathrm M_{\Sigma}f \coloneqq \mathrm M_{\Sigma,\{1\}}f$. 
\item The smooth, compactly Fourier supported version of $\mathrm M_{\Sigma,S}$,  denoted by $A_{\Sigma,S}$, and defined in \eqref{e:fourieravg}, will be used throughout the paper.   
\item Also, given $\delta>0$ we consider the $(d,n)$-Nikodym maximal operator $\mathcal N_\delta$ which is a maximal --with respect to $\sigma$-- average along plates $x+\sigma\cap B_n(1)$, $x\in\R^n$, oriented along any $\sigma\in\Grdn$ and having thickness $\delta$ in the $\sigma^\perp$-directions. %
\end{itemize}
All of the above operators are functions defined on $\R^n$ and we will be proving $L^p(\R^n)\to L^q(\R^n)$ operator norm-bounds.

\subsection{Structure of the article} In Section \ref{S:2}, we collect a few definitions related to the Grassmannian and its distance, and develop a technical subspace switch lemma for the Fourier version of our averages. Section \ref{sec:2d} uses the switch lemma to give a new and simple of the  $L^2$-almost-orthogonality principle of \cite{ASV}  for maximal directional operators in the plane. The argument of Section \ref{sec:2d} serves as a model for the more complex Section \ref{sec:algao}, where an algebraic almost orthogonality principle in arbitrary dimension, of similar flavor, is proved. Section \ref{s:5} contains the proof of Theorems \ref{t:sharpss} and \ref{t:sharp3d}. In Section  \ref{s:kakeya} we discuss the Nikodym analogue of the maximal function $\M_\Sigma$, formulating the relevant maximal conjecture and proving the $L^2$ case. Finally, Section \ref{s:cod1} is dedicated to the full solution of the codimension 1 case via subspace Carleson embedding theorem and to the application of the latter to the Rubio de Francia estimate for conical cutoffs in $\R^n$.
\section{Grassmannian, Fourier averages and switch lemmas} \label{S:2} This section contains a few definitions and technical lemmas that will be used throughout the paper.  

\subsection{Grassmannian}  We write  $\mathrm{Gr}(d,n)$ for  the Grassmannian of $d$-dimensional subspaces of $\R^n$. If $O(n)$ stands for the orthogonal group on $\R^n$ then
\begin{equation}
\label{e:qgroup}
\mathrm{Gr}(d,n) = O(n)\backslash \left[O(d)\otimes O(n-d)\right],
\end{equation}
identifying each subspace   $\sigma\in \mathrm{Gr}(d,n)$ with the orthogonal map sending  the first $d$ canonical vectors onto an orthonormal basis of $\sigma$. In particular $\mathrm{Gr}(d,n)$ is a smooth algebraic variety of dimension $d(n-d)$. Equipped with the metric
\[
\mathsf{d}(\sigma,\tau) \coloneqq \sup_{v\in \mathbb S^{n-1}} |\sigma v-\tau v|, \qquad \sigma,\tau\in \mathrm{Gr}(d,n),
\] 
the Grassmanian $\mathrm{Gr}(d,n)$ can be viewed as a compact metric space. For $\delta>0$ and $\sigma \in \mathrm{Gr}(d,n)$ we denote by $\mathsf{B}_\delta(\sigma)\coloneqq \{\tau\in \mathrm{Gr}(d,n):\, \mathsf{d}(\sigma,\tau)<\delta\}$ the open $\delta$-ball centered at $\sigma \in \mathrm{Gr}(d,n)$. 

In analogy with the classical Kakeya-Nikodym directional maximal functions, we will consider below the maximal subspace averages along $\delta$-separated subsets   $\Sigma\subset \mathrm{Gr}(d,n) $. We say $\Sigma\subset \mathrm{Gr}(d,n) $ is $\delta$-separated if 
$\{\mathsf{B}_\delta (\sigma):\sigma \in \Sigma\}$ is a collection of pairwise disjoint sets.
We will need the following lemma concerning the cardinality of the subset of a $\delta$-separated set $\Sigma$ consisting of subspaces  which are $\delta$-approximately orthogonal to some $\xi \in \R^n$. We will see that these belong to a $\delta$-neighborhood of the Grassmanian hyperplane
\begin{equation}
\label{e:hd}
{H_\xi(d)}\coloneqq \{\tau \in \mathrm{Gr}(d,n): \, \pr_\tau \xi =0\}.
\end{equation}
Notice that ${H_\xi(d)}$ is linearly isomorphic to $\mathrm{Gr}(d,n-1)$.

\begin{lemma}\label{l:grasm}  Let $\xi\in \R^n\setminus \{0\}$, $\Sigma \subset \mathrm{Gr}(d,n)$, $\delta>0 $ and  
$\displaystyle\Sigma_\xi\coloneqq \left\{\sigma \in \Sigma:\, \frac{|\pr_\sigma \xi |}{|\xi|}<\frac{\delta}{4} \right\}.$ Then
\begin{itemize}
\item[\emph{1.}] The set $\Sigma_\xi$ is contained in the  $\frac{\delta}{3}$-neighborhood of  $H_\xi(d)$. 
\item[\emph{2.}] If $\,\Sigma$ is a $\delta$-separated set, then $ \#\Sigma_\xi 
 \lesssim \delta^{-d(n-d-1)}$.
\end{itemize}
\end{lemma}

\begin{proof} We first prove claim 2 assuming claim 1. Indeed, assuming 1, for each $\sigma\in \Sigma_\xi$  we may pick $a_\sigma\in {H_\xi(d)}$ with $\mathsf{d}(a_\sigma,\sigma)<\frac\delta3$.  By the triangle inequality $\Sigma \ni \sigma,\tau,\, \sigma\neq \tau \implies \mathsf{d}(a_\sigma,a_\tau)\geq \frac{\delta}{3} $. Therefore, the set $\{a_\sigma:\sigma \in \Sigma_ \xi \}\subset {H_\xi(d)}$ has at most $\sim\delta^{-{d(n-d-1)} }$ elements by dimensionality of ${H_\xi(d)}$. This completes the proof of the lemma up to establishing the claim.

 We now prove claim 1. By rotation invariance of the statement it suffices to prove the claim for $\xi=e_1$. Let $\sigma \in \Sigma_{\xi}$. Let $u=|\pr_{\sigma}e_1|$. If $u=0$ there is nothing to prove, which means we may work with $0<u<\delta /4$. Pick an orthonormal  basis $\{b_1,\ldots, b_d\}$ of   $\sigma$ with  $b_1=\pr_{\sigma}e_1/|\pr_{\sigma}e_1|$. Then  $b_2,\ldots, b_d\in e_1^\perp\cap b_1^{\perp}$. Let
 \[
 c_1\coloneqq \frac{b_1-u e_1}{|b_1-u e_1|}.
 \]
Then $a_\sigma\coloneqq\mathrm{span}\{c_1,b_2,\ldots b_d\}\in {H_\xi(d)}$ and as $|c_1-b_1|<\frac{\delta}{3}$, we have shown that $\mathsf{d}(a_\sigma,\sigma)<\frac\delta3$ as claimed.
\end{proof}
\subsection{Fourier averages and switch lemmas} If $0\leq \delta\ll 1$ and  $\sigma \in \mathrm{Gr}(d,n)$ then
 \begin{equation}\label{e:plate}
T_\delta (\sigma)\coloneqq T_\delta ^1 (\sigma) \coloneqq \big\{\xi\in \R^{n}: \,\, \left|\pr_{\sigma}\xi\right|\leq 1, \,	 \left|\pr_{\sigma^\perp}\xi\right|<\delta \big\} 
\end{equation}
will stand for the unit scale $d$-dimensional $\delta$-plate oriented along $\sigma$. In general, we think of $T_\delta(\sigma)$ as a slight fattening of the unit ball $B_d$ on $\sigma$. Throughout the paper a few slightly different versions of this fattening will be employed depending on the problem being considered.

Let  $\phi_d \in \mathcal S(\R^{d})$ be a real valued even function with support in $2^{-8}B_{d}$, and $\|\phi_d\|_{L^1(\R^d)}=1$. For $\sigma \in \mathrm{Gr}(d,n)$ and $s>0$ we  define the smooth subspace averages and   maximal averages of $f\in \mathcal S(\R^d)$ by 
\begin{equation}
\label{e:fourieravg}
A_{\sigma,s} f(x) \coloneqq \int_{\R^n} \widehat{f}(\xi) \phi_d\left(s\pr_\sigma \xi \right)\, \e^{ix\cdot \xi}\, \d \xi, \qquad 
A_{\Sigma,S} f(x)\coloneqq\sup_{\substack{s\in S\\\sigma \in\Sigma}} |A_{\sigma,s} f(x)|,
\qquad x\in \R^n. 
\end{equation}
When $\sigma=\R v$ for some $v\in \mathbb S^{n-1}$, we write $A_{v,s}$ in place of  $A_{\sigma,s}$.  The easily verified fact 
\begin{equation}
\label{e:compar}
\|M_{\Sigma,S}\|_{L^2(\R^n)}= \sup_{k\in\mathbb Z}\|M_{\Sigma,2^kS}\|_{L^2(\R^n)}\sim\sup_{k\in\mathbb Z}\|A_{\Sigma,2^kS}\|_{L^2(\R^n)}= \|A_{\Sigma,S}\|_{L^2(\R^n)}
\end{equation}
will be used  in what follows, often in combination with \eqref{e:switch2} below.

Fix a radial function  $\Phi\in \mathcal S(\R^n)$ with 
$
\cic{1}_{B_n} \leq \Phi \leq \cic{1}_{2B_n}
$
and for $\delta>0$ introduce the  low-high splitting
\begin{equation}
\label{e:LPsplit}
A_{\sigma,s} f= A_{\sigma,s}^{>\delta} f +  A_{\sigma,s}^{<\delta} f,\qquad   
\widehat{ A_{\sigma,s}^{>\delta} f }(\xi) \coloneqq \widehat{ A_{\sigma,s}  f }(\xi)\Phi\left(   4s\delta\xi \right), \qquad \xi \in \R^n.
\end{equation}
The notation is motivated by the fact that $A_{\sigma,s}^{>\delta}$ is a smooth average at spatial scale $s$ in the directions of $\sigma$, and scale $s\delta$ in the directions orthogonal to $\sigma$, and consequently containing frequency scales at most $1/(s\delta)$. With this in mind, the parameter $0<\delta\leq 1$ in the following lemma measures the eccentricity of the plates with long directions oriented along the subspace  $\tau$ appearing in the averages on the right hand side.
\begin{lemma} \label{l:1}
Let $s>0$, $\sigma\neq \tau \in \mathrm{Gr}(d,n) $ and $\delta\in [ \mathsf{d}(\sigma,\tau), 1]$ be given. Then
\begin{align}
\label{e:switch}
& \left|A_{\sigma,s}^{>\delta}  f(x)\right|\lesssim \sum_{k=0}^\infty  2^{-kn}\fint \displaylimits _{x + 2^ksT_\delta(\tau) } \left| f \right|,
\\ 
\label{e:switchdiff} &\left|A_{\sigma,s}^{>\delta}  f(x)- A_{\tau,s}^{>\delta}f(x) \right|\lesssim \frac{\mathsf{d}(\sigma,\tau)}{\delta} \sum_{k=0}^\infty  2^{-kn}\fint \displaylimits _{x + 2^ksT_\delta(\tau) } \left| f \right|.
\end{align}
\end{lemma}

\begin{proof} The case $\sigma=\tau$ of \eqref{e:switch} follows easily by the Schwartz decay of the smooth function $\phi_d$ used in the definition of $A_{\sigma,s}$. This means that the general case of \eqref{e:switch} is an immediate consequence of \eqref{e:switchdiff}, which we now prove. By isotropic scaling and rotational invariance it suffices to treat the case $s=1$, $\tau=\mathrm{span}\{e_1,\ldots, e_d\}$. Let 
\[
\zeta(\xi,y) \coloneqq \pr_\tau \xi + y (\pr_\sigma  -\pr_\tau) \xi, \qquad y \in [0,1].
\]
The Fourier transform of the integral kernel  $K $ of $A_{\sigma,1}^{>\delta}  - A_{\tau,1}^{>\delta}$ is given by
\[
\widehat K(\xi ) =  \Phi(4\delta\xi) \int_0^1 \,\nabla \phi_d(\zeta(\xi, y)) \cdot\left[\left(\pr_\sigma  -\pr_\tau\right)  \xi  \right]\, \d y \]
and   is supported in a $1$-neighborhood of $\sigma^\perp\cap 2\delta^{-1}B_{d}$, which has measure $\sim \delta^{d-n}$. It also  
satisfies for each multi-index $\alpha=(\alpha_1,\ldots, \alpha_n)$
\begin{equation}
\label{e:cf}
\begin{split}
& \quad \big|\partial^\alpha \widehat K(\xi ) \big|   
  \lesssim  {\mathsf{d}(\sigma,\tau)} \delta^{\alpha_{d+1}+\cdots +\alpha_n-1}. 
\end{split}
\end{equation}
This estimate is obtained via repeated use of the Leibniz rule and the following considerations:
\[
\mathrm{supp}\, \widehat K \subset   2\delta^{-1}B_{d}, \qquad
\sup_{y\in [0,1]} \sup_{d+1\leq j \leq  n} \left| \partial_j \zeta(\xi,y) \right| \lesssim \mathsf{d}(\sigma,\tau)\leq  \delta, \qquad \|\pr_\sigma  -\pr_\tau\|  =  \mathsf{d}(\sigma,\tau).
\] 
Integration by parts then readily yields
\[
  |K(x)| \lesssim    \frac {\mathsf{d}(\sigma,\tau)}{\delta^{n-d+1}}  \big[  1+{|\pr_\tau x|}\big]^{-4n}\left[ 1 +   \frac{|\Pi_{\tau^\perp} x|}{\delta}    \right]^{-4n},\qquad x\in \R^n,
\]
which in turn implies  \eqref{e:switchdiff}. The proof is thus complete.
\end{proof}

\begin{remark} \label{rem:unravel}
If $\mathrm{M}$ is the standard maximal function in $\R^n$, a simple averaging argument leads from \eqref{e:switch} and \eqref{e:switchdiff} to
\begin{equation}
\label{e:switch2}
 \left|A_{\sigma,s}^{>\delta}  f(x)\right|+ \frac{\delta} {\mathsf{d}(\sigma,\tau)}\left|A_{\sigma,s}^{>\delta}  f(x) - A_{\tau,s}^{>\delta}  f(x)\right|\lesssim \sum_{k=0}^\infty  2^{-kn} \mathrm{M}[\langle| f| \rangle_{\tau,2^ks}](x), \qquad x\in \R^n.
\end{equation}
\end{remark}
The previous lemma suggests that the low frequencies $|\xi|\lesssim (s\delta)^{-1}$ of the averages $A_{\sigma,s}$ may be approximated by plates oriented along $\delta$-nearby subspaces. The next lemma records the support of the complementary high-frequency components. For $\sigma\in \mathrm{Gr}(d,n)$ and $\delta>0$, consider the two-sheeted cone
\begin{equation}\label{e:2sc}
\Gamma_{\sigma,\delta}\coloneqq \left\{\xi \in \R^n \setminus \{0\}:\, \frac{|\pr_\sigma \xi |}{|\xi|} <2^{-2}\delta\right\}
\end{equation}
and, abusing notation, denote also the corresponding Fourier restriction by $\Gamma_{\sigma,\delta}$, namely we write $(\Gamma_{\sigma,\delta} f)^\wedge (\xi)\coloneqq \ind_{\Gamma_{\sigma,\delta}}(\xi) \hat f(\xi) $. 

\begin{lemma} \label{l:2} Let $s>0,\sigma\in \mathrm{Gr}(d,n)$. Then
$   A_{\sigma,s}^{<\delta} f = A_{\sigma,s}^{<\delta}  \Gamma_{\sigma,\delta}f$.
\end{lemma}

\begin{proof} Suppose $\xi$ belongs to the Fourier support of $ A^{<\delta}_{\sigma,s}  f .$ Then $|\xi|\geq 4 (\delta s)^{-1}$, and
$
 |\pr_\sigma \xi | < {s}^{-1} \leq \delta |\xi |/4$. The latter means $\xi \in  \Gamma_{\sigma,\delta}.$
\end{proof}
\section{Almost Orthogonality for Directions in the Plane}\label{sec:2d}
In this section, we present a simple proof of the  $L^2$-almost-orthogonality principle of \cite{ASV}  for maximal directional operators in the plane. Our argument uses Fourier analysis and overlap estimates instead of geometric considerations and $TT^*$-type arguments. Although this result is known, we include here a new argument that serves as a prelude to the more technical algebraic almost orthogonality  in arbitrary dimension devised in the next section. The statement is as follows.
\begin{theorem} \label{t:2dortho} There is an absolute constant $C$ such that the following holds. Let $S\subset (0,\infty)$, $U=\{u_1,\ldots, u_{N+1}\}\subset \mathbb S^1 $ be a set of directions ordered counterclockwise. For each $j=1,\ldots, N$ let $V_j\subset \mathbb S^1$ be a set of directions contained in the cone bordered by $u_{j}, u_{j+1}$ and let $V=\bigcup\{V_j:1\leq j \leq N\}$. Then,
\begin{equation}
\label{e:v2}
\left\|\mathrm{A}_{V,S}    \right\|_{L^2(\R^2)} \leq C \left\|\mathrm{A}_{U,S}    \right\|_{L^2(\R^2)}  + \max_{1\leq j \leq N}\left\|\mathrm{A}_{V_j,S}    \right\|_{L^2(\R^2)}.
\end{equation}
\end{theorem}
\label{s:ao}
A simple induction argument using the leftmost and  middle directions as elements of the dividing set $U$, together with the control from \eqref{e:compar}, bootstrap Theorem \ref{t:2dortho} to recover the following sharp bound from \cite{Katz,ASV}.

\begin{corollary} \label{c:2dsharp}  Let $V\subset \mathbb{S}^1$ be a finite set. Then
$\displaystyle\left\|\mathrm{M}_{V,(0,\infty)}    \right\|_{L^2(\R^2)} \lesssim  \log (\# V).$
\end{corollary}

\begin{remark}\label{rmrk:2dsharp} The sharpness of the estimate in Corollary~\ref{c:2dsharp} follows by a variation of an example employed in \cite{Dem}*{Proposition 3.3} for the directional Hilbert transform. In particular one considers the action of $\mathrm{M}_{\Sigma_\delta,(0,\infty)}$ on the function $f(x)\coloneqq |x|^{-1}\ind_{\{1\leq |x|\lesssim N\}}$, where $\Sigma_\delta$ is a $\delta=N^{-1}$-net on $\mathbb S^1$. We omit the details.
\end{remark}

\begin{proof}[Proof of Theorem \ref{t:2dortho}] In this proof the constant $C>0$ is absolute and may vary at each occurrence. For each $1\leq j\leq N$, let $\Gamma_j$ be the two-sheeted cone bordered by the supporting lines to $u_j^\perp,u_{j+1}^\perp$, and denote also by $\Gamma_j$ be the corresponding Fourier restriction. 
The cones $\{\Gamma_j:1\leq j \leq n\}$ are pairwise disjoint, so that 
   \begin{equation}
\label{e:ortho2}
  \sum_{j=1}^N \left\|{\Gamma}_j f\right\|_2^2 \leq \|f\|_2^2. 
\end{equation} 
For $v\in V_j$, let $u(v)=\arg\min\left\{\mathsf{d}(v,u_j),\mathsf{d}(v,u_{j+1}) \right\}$ and $\delta(v)= {\mathsf{d}(v,u(v))}$. By assumption, the direction $v_j^\perp \in \mathbb S^1$ lies between $u_j^\perp,u_{j+1}^\perp$ and thus in $\Gamma_j$. The sense of the definitions above is that the whole cone $\Gamma_{v, \delta(v)}\coloneqq \{\xi:\,|\xi\cdot v|<\delta_v|\xi|\}$  is contained in  $\Gamma_j$, namely
\begin{equation}\label{e:ortho3}
	 \Gamma_{v, \delta(v)} \subset  \Gamma_j, \qquad v\in V_j.
\end{equation}
With this choice, Lemma \ref{l:2} and \eqref{e:ortho3} tell us that 
\begin{equation}
\label{e:ortho4}
\left|A_{v, s}^{<\delta(v)} f\right| =  \left|A_{v, s}^{<\delta(v)}  {\Gamma}_j f\right| \leq  \left|A_{v, s}^ {>\delta(v)} {\Gamma}_j f\right| +  \left|A_{v, s}   {\Gamma}_j f\right| .
\end{equation}
Applying Remark \ref{rem:unravel} twice, we obtain
\begin{equation}
\label{e:ortho4a0}
\begin{split} 
\left|A_{v, s}   f\right| &\leq \left|A_{v, s}^{>\delta(v)}  f\right| + \left|A_{v, s}^{<\delta(v)}f\right| 
 \leq  \left|A_{v, s}^{>\delta(v)} f \right|  + \left|A_{v, s}^{>\delta(v)}{\Gamma}_jf\right| +  \left|A_{v, s} {\Gamma}_j f\right| 
\\ & \leq \left|A_{v, s} {\Gamma}_j  f\right|+  C \sum_{k=0}^\infty 2^{-k}\Big(\mathrm{M}[ \langle |f|\rangle_{u(v),2^ks}] +   \mathrm{M}[ \langle |{\Gamma}_j f|\rangle_{u(v),2^ks}] \Big),
\end{split}
\end{equation}
where we remember that $\mathrm{M}$ denotes the Hardy-Littlewood maximal operator. Taking supremum over $v\in V_j, s\in S$, $1\leq j\leq N$, and subsequently taking  $L^2$  norms,  
\begin{equation}
\label{e:ortho5}
\begin{split}
\|A_{V,S }   f \|_2 &\leq \Big\| \max_{1\leq j\leq N}\left\{ A_{V_j,S} {\Gamma}_j   f\right\} \Big\|_{2} 
\\ 
&\quad+ C   \sup_{k} \left\|\mathrm{M}\circ \mathrm M_{ U,2^kS} f \right\|_{2} 
+ C\sup_{k} \bigg\|\sup_{1\leq j\leq N}\left\{\mathrm{M}\circ \mathrm M_{ U,2^kS}[ {\Gamma}_j   f]\right\}\bigg\|_2.
\end{split}
\end{equation}
Using the observation in \eqref{e:compar}
\begin{equation}
\label{e:ortho6a}
\begin{split}
 \sup_k \left\|  \mathrm{M}\circ \mathrm M_{ U,2^kS} \right\|_{2\to 2}   \leq C \sup_k\left\| \mathrm M_{ U,2^kS} \right\|_{2\to 2} \leq C \left\| A_{ U,S} \right\|_{2\to 2}, 
\end{split}
\end{equation} and similarly using the orthogonality in \eqref{e:ortho2} and a square function argument 
\begin{equation}
\label{e:ortho6b}
\begin{split}
\bigg\|\sup_{1\leq j\leq N}\left\{\mathrm{M}\circ \mathrm M_{ U,2^kS}[ \Gamma_j   f]\right\}\bigg\|_2 
 & \leq C\left\| \mathrm{A}_{ U,S}\right\|_{2\to 2}\sqrt{  \sum_{1\leq j\leq N}\left\| \Gamma_j   f \right\|_{2}^2}  \leq C \left\| \mathrm{A}_{ U,S}\right\|_{2\to 2}
\left\|  f \right\|_{2 }
\end{split}
\end{equation}
and 
\begin{equation}
\label{e:ortho6c}
\begin{split}
  \Big\| \max_{1\leq j\leq N}\left\{ A_{V_j,S} \Gamma_j   f\right\}\Big\|_{2}
 & \leq \left( \max_{1\leq j\leq N} \left\| A_{V_j,S}\right\|_{2\to 2}\right)\sqrt{ \sum_{1\leq j\leq N}\left\| \Gamma_j   f \right\|_{2}^2 } \leq \left( \max_{1\leq j\leq N} \left\| A_{V_j,S}\right\|_{2\to 2}\right) 
\left\|  f \right\|_{2 }.
\end{split}
\end{equation}
Inserting \eqref{e:ortho6a},  \eqref{e:ortho6b} and  \eqref{e:ortho6c} into estimate \eqref{e:ortho5} completes the proof of the theorem.
\end{proof}

\section{Algebraic almost orthogonality in general codimension}\label{sec:algao}
This section contains an analogue of Theorem \ref{t:2dortho} in higher dimensions, where subsets of $\mathrm{Gr}(d,n)$ are partitioned by algebraic sets. The almost orthogonality result thus obtained may then be employed in combination with polynomial partitioning to obtain a sharpening of the recent result of \cite{DPPalg} concerning averages along arbitrary directions in $\R^3$.  
 For simplicity, we restrict ourselves to the case $d=1$  below and identify $\mathrm{Gr}(1,n)$ with $\mathbb{S}^{n-1}$ in the obvious way. Higher $d$ analogues and the related polynomial partitioning theorems on $\mathrm{Gr}(d,n)$ for $d>1$ are the object of a forthcoming companion paper.

\begin{definition} Let $p\in \R[x_1,\ldots, x_{n}]$ be a degree $D$ polynomial  and $Z(p)\coloneqq\{x\in \R^n:p(x)=0\}$ be the corresponding zero set. The associated set of \emph{cells}  $\mathbf{C}(p) $ is the set of connected components of $\mathbb S^{n-1}\setminus Z(p)$. Then   $\mathbf{C}(p) $ has $\lesssim D^{n-1}$ elements,  see e.\ g.\   \cite{BarBas}.
\end{definition}

\begin{theorem} \label{t:alldortho} Let Let $p\in \R[x_1,\ldots, x_{n}]$ be a degree $D$ polynomial. For every finite set $\Sigma\subset\mathbb S^{n-1}$  and $S\subset (0,\infty)$, 
\begin{equation}
\label{e:alldortho}
\left\|A_{\Sigma,S}    \right\|_{L^2(\R^{n})} \lesssim    \sup_{\substack{U \subset \mathbb S^{n-1}\cap  Z(p) \\ \#U\leq \#\Sigma }} \left\|A_{U,S}    \right\|_{L^2(\R^{n})}  +  D^{\frac{(n-2)}2}\sup_{\mathsf{C}\in \mathbf{C}(p)}\left\|A_{\Sigma\cap \mathsf{C},S}    \right\|_{L^2(\R^{n})}. 
\end{equation}
\end{theorem}

The gain in Theorem \ref{t:alldortho} is that the set $Z\coloneqq \mathbb S^{n-1}\cap  Z(p)  $ is a $(n-2)$-dimensional algebraic variety of controlled degree and $A_{U,S}$ is better behaved when $U\subset Z$. 

\subsection{Proof of  Theorem \ref{t:alldortho}}The strategy is similar to the one we used for Theorem \ref{t:2dortho} but with a few twists. We keep using the notation  $Z= \mathbb S^{n-1} \cap Z(p)  $ and the distance $\mathsf{d}$ used here is between elements of $\mathbb S^{n-1}$. First of all we fix $\mathsf{C}\in \mathbf{C}(P)$ and $\sigma\in \mathsf{C}$. We set
\begin{equation}
\label{e:3distance}
u(\sigma) \coloneqq \arg\min\{\mathsf{d}(\sigma,u):u\in Z \}, \qquad \delta(\sigma)\coloneqq \frac{|\mathsf{d}(\sigma,u(\sigma))|}{4},
\end{equation}
and further introduce
\begin{equation}
\label{e:3cones}
 \Gamma_{\mathsf{C}}\coloneqq \bigcup_{v\in \mathsf{C} \cap \Sigma} \Gamma_{\sigma,\delta(\sigma)}, \qquad  \widehat{\Gamma_{\mathsf{C} } f}(\xi) \coloneqq \widehat{f}(\xi) \cic{1}_{\Gamma_\mathsf{C} }(\xi),\qquad \xi\in\R^n.
\end{equation}
Note that we are again conflating the set $\Gamma_\mathsf{C} $ with the corresponding Fourier restriction.
We also define the sets
\[
U\coloneqq \bigcup_{\mathsf{C}\in \mathbf{C}(P)} U(\mathsf{C}), \qquad  U(\mathsf{C}) \coloneqq \left\{ u(\sigma):\, \sigma \in \mathsf{C}\right\}.
\]
Clearly we have that $U\subset Z$ and $\#U\leq \#\Sigma.$ 
An immediate though important geometric observation is made in the following lemma. For $\xi \in \R^n\setminus \{0\}$, recall the notation ${H_\xi(1)}=\{\tau \in \mathbb S^{n-1} : \, \Pi_\tau \xi =0\}$.

\begin{lemma} \label{l:cross} 
Let $\mathsf{C}\in \mathbf{C}(P)$ and $\xi \in \Gamma_{\mathsf{C}}$. Then $  \mathsf{C}\cap {H_\xi(1)} \neq \varnothing $.
\end{lemma}

\begin{proof}As $\Gamma_{\mathsf{C}}$ is a union of cones, it suffices to work with $\xi \in \Gamma_{\mathsf{C}}\cap \mathbb S^{n-1}$. By definition of $\Gamma_{\mathsf{C}}$ we may find $\sigma \in  \mathsf{C}$ such that $|\Pi_\sigma \xi|<\delta(\sigma)$.  The first claim of Lemma \ref{l:grasm} tells us that  there exists  $\tau \in {H_\xi(1)}$ with $\mathsf{d}(\sigma,\tau)<\delta(\sigma)/3$. However the  set $\{\tau \in \mathbb S^{n-1}: \mathsf{d}(\sigma,\tau)<\delta(\sigma) \}$ is contained in $\mathsf{C}$ by the definition of $\delta(\sigma)$. 
\end{proof} 

Lemma \ref{l:cross} is the main cog in the proof of the following square function estimate.

\begin{lemma} \label{l:squarefalld}  $\displaystyle  \sum_{\mathsf{C}\in \mathbf{C}(P)} \left\|\Gamma_{\mathsf{C} } f\right\|_2^2\lesssim D^{n-2}\left\|f\right\|_2^2.$
\end{lemma}

\begin{proof}
It suffices to show that for almost every $\xi \in \mathbb S^{n-1}$
\[
\#\mathbf{C}(\xi)\lesssim D^{n-2}, \qquad \mathbf{C}(\xi)\coloneqq\left\{\mathsf{C}\in \mathbf{C}(P): \xi \in   \Gamma_{\mathsf{C}}\right\}. 
\]
Lemma \ref{l:cross} tells us that if $\mathsf{C}\in\mathbf{C}(\xi)$ then $\mathsf{C}\cap {H_\xi(1)}$ is a connected component of ${H_\xi(1)}\setminus Z.$ {As $\mathrm{dim}\, {H_\xi(1)}= n-2$  this set has  $\lesssim D^{n-2}$ connected components, see e.\ g.\ \cite[Theorem 4.11]{Sheffer} or \cite{BarBas}}.
\end{proof}

We begin the main argument. Fixing $\mathsf{C}\in \mathbf{C}(p), \sigma\in \mathsf{C} $, we apply Lemma \ref{l:2} and the definition of $\Gamma_{\mathsf{C}}$ to obtain
\begin{equation}
\label{e:ortho4bis}
\Big|A_{\sigma, s}^{<\delta(\sigma)}  f\Big| =  \left|  A_{\sigma, s}^{<\delta(\sigma)}\Gamma_{\mathsf{C}} f\right| \leq C \mathrm{M}[  A_{\sigma, s} \Gamma_{\mathsf{C}} f],
\end{equation}
which together with Lemma \ref{l:1} yields
\begin{equation}
\label{e:ortho4a0bis}
\begin{split}   
\left|A_{\sigma, s}   f\right|  \leq \left|A_{\sigma, s}^{>\delta(\sigma)}   f\right| + \left|A_{\sigma, s}^{<\delta(\sigma)}   f\right| 
\leq C \mathrm{M}\circ \mathrm{M}_{u(\sigma),s} f + C\mathrm{M}[  A_{v, s} \mathrm{\Gamma}_{\mathsf{C}} f].
\end{split}
\end{equation}
We first take supremum over $s\in S, \sigma\in \mathsf{C}$ and obtain 
\begin{equation}
\label{e:ortho4abis}
\begin{split}
A_{\Sigma\cap \mathsf{C}, S }   f     \leq C  \mathrm{M}\circ \mathrm{M}_{U(\mathsf{C}), S} f + C
\mathrm{M}\circ A_{\Sigma\cap\mathsf{C}} (\mathrm{\Gamma}_{\mathsf{C}} f).  
\end{split}
\end{equation}
Subsequently taking supremum over $\mathsf{C}\in \mathbf{C}(P) $ leads to
\begin{equation}
\label{e:ortho5bis}
\begin{split}
A_{\Sigma ,S}   f  \lesssim   \mathrm{M}\circ \mathrm{M}_{ U} f + \left(
 \sum_{\mathsf{C} \in \mathbf{C}(p)}\left| \mathrm{M}\circ A_{\Sigma\cap\mathsf{C}} (\mathrm{\Gamma}_{\mathsf{C}} f) \right|^2\right)^{\frac12}.
\end{split}
\end{equation}
The estimate of Theorem \ref{t:alldortho} then follows easily from \eqref{e:ortho5bis} and the square function estimate of Lemma \ref{l:squarefalld}.

\section{Sharp or nearly sharp bounds for maximal subspace averages}  \label{s:5}
We now focus on the single scale maximal operator $\mathrm{M}_{\Sigma,\{1\}}\eqqcolon\mathrm{M}_{\Sigma} $ when $\Sigma$ is a generic finite subset of $\mathrm{Gr}(d,n)$.  
The majority of this section is in fact dedicated to the proof of Theorem \ref{t:sharpss}. However, we first detail the announced sharpening of the case $n=3$, $d=1$.
In this case, Theorem~\ref{t:sharpss} tells us that
\[
\sup_{\substack{\Sigma\subset \mathbb S^{2}\\  \#\Sigma \leq N} } \left\|\mathrm{M}_{\Sigma} \right\|_{L^2(\R^3)} \lesssim N^{\frac{1}{4}} \log N
\]
which is sharp up to the logarithmic factor and recovers the bound from \cite{Dem12}. The following  more precise estimate was previously proved in \cite{DPPalg}: for any positive integer $k$ we have
\[
\sup_{\substack{\Sigma\subset \mathbb S^{2}\\  \#\Sigma \leq N} } \left\|\mathrm{M}_{\Sigma} \right\|_{L^2(\R^3)} \lesssim_k N^{\frac{1}{4}} (\log N)^{\frac12} \log^{[k]}N,
\]
where $\log^{[1]}N\coloneqq \log(2+N)$, $\log^{[k]}N\coloneqq \log(2+\log^{[k-1]}N)$. For this special case of dimension and codimension we improve the result of the theorems above by exploiting the algebraic almost orthogonality principle of \S\ref{sec:algao}. This is the content of the following theorem.

\begin{theorem}\label{t:sharp3d} For all $N>0$ and every positive integer $k$ there holds
\[
\sup_{\substack{\Sigma\subset \mathbb S^{2}\\  \#\Sigma \leq N} } \left\|\mathrm{M}_{\Sigma} \right\|_{L^2(\R^3)} \lesssim_k N^{\frac{1}{4}} \log^{[k]}N.
\]
with implicit constant depending only on $k$. This bound is  sharp in terms of $N$ up to the iterated logarithmic factor.
\end{theorem}
We will prove  the estimates of both Theorems \ref{t:sharpss} and \ref{t:sharp3d} for the corresponding  norm-equivalent Fourier operator $A_{\Sigma,\{1\}}\eqqcolon A_{\Sigma} $. 
\subsection{Clusters on $\mathrm{Gr}(d,n)$} We begin with a definition related to $H_\xi(d)$ from \eqref{e:hd}. We say that $\Sigma$ is a $\delta$-cluster with top $\xi\in \mathbb S^{n-1}$ if $\Sigma$ is  a finite set contained in the $\delta$-neighborhood of $H_\xi(d)$.  
Somewhat dual to $\delta$-clusters are the cone cutoffs
\[ 
\Gamma_{\sigma,\delta}= \left\{ \eta \in \R^n\setminus \{0\}: |\pr_{\sigma} \eta |<2^{-4}\delta|\eta| \right\}, \qquad \Gamma_{\sigma,\delta} f(x)=\int_{\Gamma_{\Sigma,\delta}} \widehat f(\eta) \e^{i x\cdot \eta} \, \d \eta. 
\]
We summarize   two key steps of our proof in the estimate of the next lemma.

\begin{lemma}  \label{l:doubleann} Suppose  $f\in \mathcal S(\R^d)$ with 
\[
\mathrm{supp}\,  \widehat f\subset \mathrm{Ann}(\delta) \coloneqq \left\{ \eta\in \R^n:\,2^{-4}<\delta|\eta|<2^{-2}\right\}.
\]
Then
\vskip1mm \noindent \emph{1.} For all $\sigma \in \mathrm{Gr}(d,n)$ and $s\geq 1$, we have $A_{\sigma, s} f= A_{\sigma,s} {\Gamma_{\sigma,\delta}} f $.
\vskip1mm \noindent \emph{2.} If $\,\Sigma $ is a $\delta$-cluster, we have the estimate
\[\| A_{\Sigma,\{1\}} f  \|_{2} \lesssim \left( \sup_{\substack{\Sigma'\subset \mathrm{Gr}(d,n-1)\\ \#\Sigma' \leq \#\Sigma } }\| A_{\Sigma',\{1\}}   \|_{L^2(\R^{n-1})} \right) \|f\|_2.\]
\end{lemma}

\begin{proof} First of all we dispense with the support claim. As $\phi_d$ in the definition of $A_{\sigma,s}$ is supported in $2^{-8} B_d$ we have that
$\widehat {A_{\sigma,s} f}(\eta) =0$ unless $|\pr_{\sigma} \eta |<2^{-8}s^{-1}$ and $\delta|\eta|>2^{-4}$, whence the claim.

We then move to the proof of the estimate. Suppose that $\Sigma$ is a $\delta$-cluster with top $\xi\in \mathbb S^{n-1}$. For $\sigma\in \Sigma$ let $\tau(\sigma)\in H_\xi(d)$ be such that $\mathsf{d}(\sigma,\tau)<\delta$. As $\widehat f(\eta)=0$ for $\delta|\eta|\geq 2^{-2}$, we may insert the $\Phi(2^{2}\delta\cdot)$ cutoff for free and  $A_{\sigma,1} f = A_{\sigma,1}^{>\delta} f$.  Then using \eqref{e:switch} from Lemma \ref{l:1}  and Remark~\ref{rem:unravel} returns
\[
A_{\sigma,1} f \lesssim \sum_{k=0}^\infty 2^{-kn} \mathrm M\left( \langle f \rangle_{\tau(\sigma),2^k} \right)
\]
which coupled with the  norm  equivalences of \eqref{e:compar} tells us that
\[
\| A_{\Sigma,\{1\}} f  \|_{2} \lesssim \left( \sup_{\substack{\Sigma'\subset H_\xi(d)\\ \#\Sigma' \leq \#\Sigma } }\| A_{\Sigma',\{1,\}}   \|_{L^2(\R^n)} \right) \|f\|_2.
\]
 The conclusion of the lemma  then follows from recalling that $H_\xi(d)\approxeq \mathrm{Gr}(d,n-1)$ and applying Fubini's theorem in the $\xi$ direction. 
\end{proof} 

\subsection{Counterexamples yielding sharpness of Theorem~\ref{t:sharpss}} In this subsection we discuss the sharpness of the estimate in Theorem~\ref{t:sharpss}. In fact we will prove the more general proposition below.

\begin{proposition}\label{prop:cex} Let $n\geq 2$ and $1<d \leq n-1$. There holds
	\[
	\sup_{\substack{\Sigma\subset \mathrm{Gr}(d,n)\\\#\Sigma\leq N}} \|\M_{\Sigma}\|_{L^p(\R^n)}\gtrsim 
	\begin{cases}
		N^{\frac{n-d+1-p}{p(n-d)}},& \quad 1<p<n-d+1,
		\\
		(\log N)^{\frac{1}{p}},& \quad p\geq n-d+1.
	\end{cases} 
	\]
\end{proposition}

\begin{proof}[Proof of Proposition~\ref{prop:cex} for $d=1$] This case is classical and well understood but we include it here as it is instructive for case of   $2\leq d<n$ that follows. Indeed we just need to consider $f$ to be the indicator function of the unit ball in $\R^n$ and $\Sigma$ be a $\eps\delta$-net in the unit sphere with $\eps$ a numerical constant to be chosen momentarily. Clearly $\#\Sigma\eqsim_\eps \delta^{-(n-1)}$ and note that for $x\in\R^n$ and $v(x)\coloneqq x/|x|$ we have that
\[
A_{v(x),\delta^{-1}} ^{>\delta} f(x)\gtrsim \delta,\qquad x\in B_n(\delta^{-1})\setminus B_n(0,1).
\]
On the other hand for every $x\in\R^n$ there exists $v\in \Sigma$ with $\dist(v,v(x))\leq \eps \delta$ so we get by Lemma~\eqref{l:1}
\[
A_{v,\delta^{-1}} ^{>\delta} f(x) \geq c \delta - C\eps \sum_{k=0} ^\infty 2^{-kn}\fint_{x+2^k\delta^{-1}T_\delta}|f|\gtrsim \delta
\]
if $\eps$ is chosen to be sufficiently small. This readily implies
\[
\|M_{\Sigma,\{\delta^{-1}\}} \|_{L^p(\R^n)}\gtrsim \|A_{\Sigma,\{\delta^{-1}\}}\|_{L^p(\R^n)}\gtrsim \delta^{-\frac{n-p}{p}},\qquad 1<p<n. 
\]
Since $\#\Sigma =\delta^{-(n-1)}=N$ this proves the proposition when $d=1$ and $p<n$. 

For the case $d=1$, $p\geq n $ the lower bound of the order $(\log N)^{ 1/p}$ follows by the Besicovitch construction of a Kakeya set as in \cite{FeffBall}. Briefly, one constructs a set $K_\delta=\bigcup T_\delta\subset \R^2$, the $\delta$-neighborhood of a Kakeya set, which is a union of $\delta \times 1$ tubes pointing along the directions $\Sigma$ of a $\delta$-net on $\mathbb S^1$, and so that $|K_\delta|\lesssim (\log N)^{-1}$. Then it is simple to check that 
\[
\M_{\Sigma,\{3\}}\ind_{K_\delta} \gtrsim 1 \quad\text{on the set}\quad K_\delta ^*\coloneqq \bigcup T^* _\delta\setminus K_\delta,
\]
where $T^* _\delta$ is the tube with the same center and direction as $T_\delta$ and 3 times the length. It follows that $|K_\delta ^*|\eqsim 1$ and so 
\[
\|\M_{\Sigma,1}\|_{L^p(\R^2)}\gtrsim (\log N)^{\frac 1p}.
\]
For general $n\geq 2$ we just test the maximal operator on the tensor product of the Kakeya-type set above with a unit cube in $\R^{n-2}$, namely $K_\delta \times [-1/2,1/2]^{n-2}$ and the lower bound $\|\M_{\Sigma,1}\|_{L^p(\R^n)}\gtrsim (\log N)^{\frac 1p}$ follows.
\end{proof}

\begin{proof}[Proof of Proposition~\ref{prop:cex} for $1<d<n$] In general we consider $(d,n)$ plates of scale $M\coloneqq N^{1/(n-d)}$ and thickness $0$ as follows. If $\sigma \in \mathrm{Gr}(d,n)$, then
\[
T_0 ^M (\sigma)\coloneqq \{x\in \R^n: \,\, |x|< M,\, \pr_{\sigma^\perp} x=0\}
\]
is the scale $M$ plate oriented along $\sigma$.  Let $\omega\in \mathrm{Gr}(d-1,n)$ that will remain fixed throughout the argument below. Then we define 
\[
E=E(\omega) \coloneqq \{\eta\in \mathrm{Gr}(1,n): \, \Pi_\omega v=0 \;\forall v\in \eta\}\sim \mathrm{Gr}(1,n-d+1)\sim \mathbb S^{n-d}.
\]
Notice that $\mathrm{span}\,\{\omega,\eta\}\in \mathrm{Gr}(d,n) $ for all $\eta \in E$.  We may then write every $x\in \R^n$ as $x=\Pi_\omega x + \rho \eta$ for $\eta \in E\sim \mathbb S^{n-d}$ and $\rho\geq 0$. We use the notation $x=(x_\omega, \rho \eta)$ accordingly where $x_\omega=\Pi_\omega x $.
We also set 
\[
C_M=C_M(\omega)\coloneqq \{x=(x_\omega, \rho \eta)\in \R^n:\,\,|x_\omega| \leq M,\, |\rho|\leq 1 \}.
\]

\begin{lemma} Let $\omega\in \mathrm{Gr}(d-1,n)$ and $x=(x_\omega, \rho \eta)$ with $|x_\omega|\leq \frac{M}{2}$, $2^{-8} M\leq \rho \leq 2^{-7} M $. Suppose $v\in E$  with $|v-\eta|<\frac{2^{-8}}{M} $. Let $\sigma=\mathrm{span}(\omega,\eta)\in \mathrm{Gr}(d,n)$. Then
\[
\left|[x+T_0 ^M (\sigma)] \cap C_M \right| \geq 2^{-10} M^{d-1}.
\]
\end{lemma}

\begin{proof} By a rotation we may assume $\omega=\{e_1,\ldots, e_{d-1}\}$. Choose $\beta$ so that $\{\eta,\beta\}$ is an orthonormal basis of $\mathrm{span}\{\eta,v\}$. Then from the assumption of $|v-\eta|<2^{-8}M^{-1}$ we obtain that  
\[  v=(\cos \theta) \eta + (\sin \theta) \beta, \quad |\theta|\leq \frac{2^{-7}}{M}\]
so that $x=(x_\omega, \rho \eta)$ and write the generic point on the plate  $x+T_0 ^M(\sigma)$ as
\[
y=( x_\omega + t_1 e_1 + \ldots + t_{d-1} e_{d-1}, (\rho+ t_d\cos \theta) \eta +t_d(\sin \theta) \beta), \qquad t_1,\ldots,t_d \in (-M, M).
\]
Note that
\begin{equation}
\label{e:cex1}
\sup_{j} |t_j| <\frac{M}{2(d-1)} \implies
|y_\omega| \leq |x_\omega | + (d-1)\sup_{j} |t_j| < M.
\end{equation}
Also,  
\begin{equation}
\label{e:cex2}
t_d \in  I\coloneqq \left[-\frac{\rho}{\cos\theta} +\frac{1}{8}, -\frac{\rho}{\cos\theta} +\frac{1}{4},\right] \implies 2^{-10} M\leq|t_d| \leq 2^{-6} M
 \end{equation}
 so that
 \begin{equation}
\label{e:cex3}
\quad  \left|(\rho+ t_d\cos \theta) \eta +t_d(\sin \theta) \beta\right| \leq  \frac{1}{4} + 2^{-6} M |\sin \theta| \leq \frac{1}{2}.
\end{equation}
It follows that $y\in C_M $ in the range $|t_1|,\ldots |t_{d-1}|<\frac{M}{2(d-1)}, t_d\in I$ which is a set of measure $\geq 2^{-10}M^{d-1}$. 
\end{proof}
We continue with the proof of Proposition~\ref{prop:cex}. Let $E_M$ be a $2^{-18} M$-net in $E\sim \mathbb S^{n-d}$ and consider  the set
\[
\Sigma_M\coloneqq \{\sigma=\mathrm{span}(\omega, v):\, v\in E_M\} \subset \mathrm{Gr}(d,n)
\]
which has $\sim M^{n-d}=N$ elements. The maximal function
\[
\mathrm{M}_{\Sigma_M, \{M\}} f(x)= \sup_{\sigma \in \Sigma_M} \frac{1}{M^d}\int_{x+ T_0 ^M (\sigma)} |f|
\]
then satisfies
\[
\mathrm{M}_{\Sigma_M, \{M\}} \cic{1}_{C_M}(x) \gtrsim \frac{M^{d-1}}{M^d} \sim M^{-1}, \qquad x\in U_M\coloneqq\{(x_\omega, \rho \eta):\,\,|x_\omega|\leq 2^{-1}M,\, 2^{-8} M\leq \rho \leq 2^{-7} M\},
\]
whence
\[
\|\mathrm{M}_{\Sigma_M, \{M\}}\|_{L^p(\R^n)} \geq  
\frac{\|\mathrm{M}_{\Sigma_M, \{M\}} \cic{1}_{C_M}\|_p}{|C_M|^{\frac1p}} \gtrsim M^{-1+\frac{1-d}{p}} |U_M|^{\frac1p} \gtrsim M^{\frac{n-d+1}{p}-1}=N^{\frac{n-d+1-p}{p(n-d)}}.
\]
This proves the desired bound for $1<d<n$ and $p<n-d+1$. 

For the case $p\geq n-d+1$ we repeat the proof above but replacing the ball with a Kakeya set. More precisely, fixing $\omega\in\mathrm{Gr}(d-1,n)$ we let $K_\delta\subset \R^2 \subseteq \omega^\perp $ be a $\delta$-neighborhood of a Kakeya set as before, namely $K_\delta$ is a union of $\delta\times 1$ tubes pointing along a $\delta$-net  $\Sigma_\delta \subset \mathbb S^1$. Note that this is always possible as $\omega^\perp$ always contains a copy of $\R^2$. Remembering that $x_\omega=\Pi_\omega x$ we modify the definition of $C_N$ as follows
\[
C_\delta=C_\delta(\omega)\coloneqq \{x\in \R^n:\,\,|x_\omega| \leq 1,\, x-x_\omega \in K_\delta  \}.
\]
It is then easy to check that
\[
\|\mathrm{M}_{\Sigma_M, \{M\}}\|_{L^p(\R^n)} \geq  
\frac{\|\mathrm{M}_{\Sigma_M, \{M\}} \cic{1}_{C_\delta}\|_p}{|C_\delta|^{\frac1p}} \gtrsim (\log\delta^{-1})^{\frac1p} \eqsim (\log\#\Sigma_\delta)^{\frac1p}
\]
as desired.
\end{proof}

\subsection{Proof of the upper bound in Theorem \ref{t:sharpss}}   We  seek an inductive estimate for 
\[
K_{d,n}\coloneqq \sup_{N>0} \frac{1}{N^{ \frac{n-d-1}{n-d}} (\log N)^{\alpha(d,n)}}  \sup_{\substack{\Sigma\subset \mathrm{Gr}(d,n)\\ \#\Sigma \leq N}} \left\|A_\Sigma \right\|_{L^2(\R^n)}^2.
\]
where $\alpha(d,n)$ will be determined along the induction argument.  The first step in the reduction is a classical use of the Chang-Wilson-Wolff inequality. For similar applications in the setting of directional singular integrals see for example \cites{Dem,DPGTZK}.

\begin{lemma} \label{l:cww}
There holds  $ K_{d,n}\lesssim_n Q_{d,n}$, where
 \[
 Q_{d,n}\coloneqq \sup_{N>0} \frac{1}{N^{\frac{n-d-1}{n-d}} (\log N)^{\alpha(d,n)-1}}  \sup_{0<\delta <1}   \sup_{\substack{\Sigma\subset \mathrm{Gr}(d,n)\\ \#\Sigma \leq N}} \sup_{\substack{\|f\|_2=1\\ \mathrm{supp}\,  \widehat f\subset \mathrm{Ann}(\delta)  }} \left\|A_\Sigma f \right\|_{L^2(\R^n)}^2,
 \]
 with $\mathrm{Ann}(\delta)$ as defined in Lemma~\ref{l:doubleann}.
\end{lemma}

The induction parameter is $n$, while $d$ is kept fixed along the induction. The seed for the induction is the base case $n=d+1$.

\begin{lemma}\label{l:baselemma} If $\alpha(d,d+1)=1$, then $  Q_{d,d+1}\lesssim_d 1$.
\end{lemma}

The proof of Lemma \ref{l:baselemma} is given at the end of this section. Notice that the lemma, together with Lemma \ref{l:cww}, implies $K_{d,d+1} \lesssim 1
$ with the choice $\alpha(d,d+1)=1$.
Fix now $N>0$, $0<\delta<1$, $\Sigma\subset \mathrm{Gr}(d,n)$ with $ \#\Sigma \leq N$, $f\in \mathcal S(\R^n) $ with $\|f\|_2=1$ and  $\mathrm{supp}\,  \widehat f\subset \mathrm{Ann}(\delta)$. Using a greedy selection algorithm, we may achieve 
\[
\Sigma = \Sigma_0 \cup \bigcup_{j=1}^{  \Theta} \Sigma_j,
\]
where $0\leq 
\Theta\leq N^d$, each $\Sigma_j$ is a $\delta$-cluster and $\Sigma_0$ has the property that
\begin{equation}
\label{e:ovclust}
\sum_{\sigma\in \Sigma_0} \cic{1}_{\Gamma_{\sigma,\delta}} (\eta) \leq N^{\frac{n-d-1}{n-d}} \qquad \forall \eta \in \R^n\setminus \{0\}.
\end{equation} 
Say $\xi \in \mathbb S^{n-1}$ is bad for $\Sigma'$ if  the set $\{\sigma \in \Sigma' : \xi \in \Gamma_{\sigma,\delta} \}$ has cardinality $>N^{\frac{n-d-1}{n-d}}.$
The first step in the greedy selection algorithm is to initialize and set $\Sigma'=\Sigma$. If some bad  $\xi $ exists for $\Sigma'$, set $\Sigma_1\coloneqq \{\sigma \in \Sigma' : \xi \in \Gamma_{\sigma,\delta} \}$. Notice that by Lemma \ref{l:grasm} the set $\Sigma_1$ is a $\delta$-cluster. Now set  $\Sigma'\coloneqq \Sigma'\setminus  \Sigma_1$  and repeat. The algorithm terminates when  no bad $\xi$ exists for $\Sigma'$, in which case set $\Sigma_0=\Sigma'$. Notice that \eqref{e:ovclust} then holds by construction. Cardinality considerations tell us that the algorithm terminates after $\Theta\leq N^d$ steps. Using the first claim of Lemma \ref{l:doubleann}, we get at once
\begin{equation}
\label{e:goodclust}
\left\|A_{\Sigma_0} f\right\|_2^2 \leq \sum_{\sigma \in \Sigma_0} \left\|A_{\Sigma_0} \Gamma_{\sigma,\delta} f\right\|_2^2 \lesssim \sum_{\sigma\in \Sigma_0} \cic{1}_{\Gamma_{\sigma,\delta}} \leq N^{ \frac{n-d-1}{n-d}}.
\end{equation}
Suppose that the cluster $\Sigma_j$ has $\sim 2^{k} N^{ \frac{n-d-1}{n-d}}$ elements for some $1\leq 2^{k}\leq N^\frac{1}{n-d}$. Then the second claim of Lemma  \ref{l:doubleann} tells us that
\begin{equation}
\label{e:badclust}
\left\|A_{\Sigma_j} f\right\|_2^2 \lesssim K_{d,n-1} 2^{k \frac{n-d-2}{n-d-1}} N^{\frac{n-d-2}{n-d}}  (\log N)^{\alpha(d,n-1)}
\end{equation}
Notice that $\#\{j:\Sigma_j\sim  2^{k} N^{ \frac{n-d-1}{n-d}}\}\leq 2^{-k}N^{\frac{1}{n-d}}$, so that
\begin{equation}
\label{e:badclust2}
\begin{split}
&\quad \sum_{j=1}^{\Theta}
\left\|A_{\Sigma_j} f\right\|_2^2 \lesssim K_{d,n-1} N^{\frac{n-d-1}{n-d}}	 (\log N)^{\alpha(d,n-1)}  \sum_{k=1}^{\infty}   2^{k \frac{n-d-2}{n-d-1}-k}\\ &\lesssim K_{d,n-1} N^{\frac{n-d-1}{n-d}} (\log N)^{\alpha(n-1)}.
\end{split}
\end{equation}
Combining \eqref{e:goodclust} and \eqref{e:badclust2} we see that $
Q_{d,n}\lesssim K_{d,n-1}
$, so that $K_{d,n}\lesssim K_{d,n-1}$ provided that $\alpha(d,n)=\alpha(d,n-1)+1$. Induction completes the proof   with $\alpha(d,n)=n-d.$

\subsection{Proof of Lemma \ref{l:baselemma}}  We have $n=d+1$ throughout the proof of the lemma. We  use below the the slight enlargement of $\mathrm{Ann}(\delta)$ from the statement of Lemma~\ref{l:doubleann}
\[
\mathrm{Ann}^+(\delta) \coloneqq \left\{ \eta\in \R^n:\, 2^{-5}<\delta|\eta|<2^{-1}\right\}\supset \mathrm{Ann}(\delta).
\]
Before the proof proper we carefully reshuffle the conclusion of Lemma \ref{l:1}. We say that $m\in \mathcal S(\R^{n})$ is $\delta$-adapted to  $\tau=\mathrm{span}\{e_1,\ldots,e_d\}$ if 
\begin{itemize}
\item[a.] $\mathrm{supp}\, m \subset O_{\tau,\delta}\coloneqq\{\xi \in \mathrm{Ann}^+(\delta)  :\, |\pr_\tau \xi|\leq 2\delta |\xi|\}$,
\item[b.]  $\|\partial^\alpha m\|_\infty \leq \delta^{\alpha_{d+1}}$ for all multi-indices $\alpha=(\alpha_1,\ldots,\alpha_{d+1})$ of order $|\alpha|\leq 100d $.  
\end{itemize}
 Now if $\tau \in \mathrm{Gr}(d,d+1)$ is generic, we say that $m$ is $\delta$-adapted to $\tau$ if $m\circ \tau$ is $\delta$-adapted to $\mathrm{span}\{e_1,\ldots,e_d\}$, where $\tau$ also stands for the rotation mapping $\mathrm{span}\{e_1,\ldots,e_d\}$ to $\tau$ and $e_{d+1}$ to $\tau^\perp$. A typical example of function $\delta$-adapted to $\tau$ is 
\[
m(\xi)\coloneqq\mu(\pr_\tau \xi) \Psi(\delta \xi), \qquad \xi\in \R^{d+1},
\]
where $\Psi\in\mathcal S(\R^{d+1})$ is supported in $\mathrm{Ann}(1)$ and $\mu\in\mathcal S(\R^{d})$ is supported on a small ball near the origin and suitably normalized. 
 We denote by $\mathsf{m}(\tau,\delta)$ the class of  multipliers which are $\delta$-adapted to $\tau$ and define
\[
\mathsf{A}^\delta_{\tau} f(x) \coloneqq \sup_{m\in \mathsf{m}(\tau,\delta)}\left| \int_{\R^{d+1}} \widehat f(\xi) m(\xi) \e^{i x\cdot \xi
} \, \d \xi \right|, \qquad x\in \R^{d+1}.
\]
A repetition of the proof of \eqref{e:switch} of Lemma \ref{l:1} tells us that $\mathsf{A}^\delta_{\tau} f$ is pointwise bounded by the right hand side of  \eqref{e:switch} for $s=1$. This together with frequency support considerations yield
\begin{equation}
\label{e:singlecone}
\|\mathsf{A}^\delta_{\tau} f\|_{2}  \lesssim \|O_{\tau,\delta} f\|_2
\end{equation}
uniformly in $\tau\in \mathrm{Gr}(d,d+1)$ and $\delta>0$, where we have denoted by $O_{\tau,\delta}$ the frequency cutoff to the corresponding conical sector. Now, if $\Sigma_\delta$ is  $2^{-10}\delta$-net in $\mathrm{Gr}(d,d+1)$, we have
\begin{equation}
\label{e:deltaortho}
 \sum_{\tau \in  \Sigma_\delta} \|O_{\tau,\delta} f\|_2^2 \lesssim \|  f\|_2^2
 \end{equation}
as the projection on $\mathbb S^{d}$ of  $O_{\tau,\delta} $ lies in a $\sim \delta$-neighborhood of $\tau$. 
A square function argument combining   \eqref{e:singlecone} with  \eqref{e:deltaortho} then yields
\begin{equation}
\label{e:uninet}
\bigg\| \sup_{\tau \in  \Sigma_\delta} \mathsf{A}^\delta_{\tau} f \bigg\|_2 \lesssim \|  f\|_2.
 \end{equation} 
 We are ready to complete the proof proper of this Lemma. Fix $\delta>0$ and let $f\in L^2(\R^{d+1})$   with $\mathrm{supp} \, \widehat f \subset \mathrm{Ann}(\delta)$.
 Read from the proof of Lemma \ref{l:1}, cf.\ \ref{e:cf} in particular,  that 
\begin{equation} \label{e:poibd}
\sup_{\substack{ \sigma \in \mathrm{Gr}(d,d+1) \\ \mathsf{d}(\sigma,\tau)<\delta }} \left |A_{\sigma, 1} f\right| \lesssim  \mathsf{A}^\delta_{\tau} f 
\end{equation}
by means of a suitable insertion of a Littlewood-Paley cutoff equal to $1$ on $\mathrm{Ann}(\delta)$ and supported on $\mathrm{Ann}^+(\delta) $.  As $\Sigma_\delta$ is a $2^{-10}\delta$-net, an application of \eqref{e:poibd} followed by \eqref{e:uninet} returns
\begin{equation}
\label{e:baselemma1}
\bigg\|
\sup_{\sigma\in \mathrm{Gr}(d,d+1)} | A_{\sigma, 1} f | \bigg\|_{L^2(\R^{d+1})} \lesssim \bigg\| \sup_{\tau \in  \Sigma_\delta} \mathsf{A}^\delta_{\tau} f \bigg\|_2 \lesssim \|  f\|_2
\end{equation}
which is what is claimed in the lemma.

\subsection{Polynomial partition for the proof of Theorem~\ref{t:sharp3d}} We plan to apply the algebraic almost orthogonality principle of \S\ref{sec:algao} in order to prove Theorem~\ref{t:sharp3d}. The first order of business is to feed the almost orthogonality result of Theorem~\ref{t:alldortho} with a suitable polynomial partition of the set $\Sigma\subseteq \mathbb S^2$ with $\#\Sigma\leq N^2$, tailored to the problem in hand.

\begin{proposition}\label{prp:pp} Let $\Sigma \subset \mathbb S^2$ be a finite set with $\#\Sigma \leq N$ and let $\delta>0$. For each integer $D\geq 2^3$ there exists a partition
\[
\Sigma =\Sigma_{\circ}\cup\Sigma_{\times}	
\]
satisfying the following properties:
\begin{itemize}
	\item [1.] There exist $O(1)$ transverse complete intersections $W_j\subset \mathbb S^2$ of dimension $1$ and degree $O(D)$ such that
	\[
	\sup_{\sigma\in \Sigma_\times }\inf_j \mathrm{dist}(\sigma,W_j)<\delta.
	\]
	\item[2.] There exist $O(D^2)$ disjoint connected open subsets $\mathsf{C}\in\vec{\mathsf C}$ of $\,\mathbb S^2$ with the property that
	\[
	\Sigma_\circ=\bigcup_{\mathsf{C}\in\vec{\mathsf{C}}}\Sigma_{\mathsf C},\qquad \Sigma_{\mathsf C}\coloneqq \Sigma\cap\mathsf{C},\qquad \#\Sigma_{\mathsf C}\leq \frac{N}{D^2}.
	\]
\end{itemize}
\end{proposition}
The proposition above is a consequence of the more general polynomial partitioning result of \cite{DPPalg}*{Proposition 2.10}. We also refer the reader to \cite{DPPalg}*{\S2.9} for the definition of a \emph{transverse complete intersection} and further background on polynomial partitioning results. 

The partition of Proposition~\ref{prp:pp} is not directly applicable as an input for Theorem~\ref{t:alldortho} as the set $\Sigma_\times$ also contains directions \emph{close} to the algebraic variety $Z\coloneqq \cup_j W_j$ instead of just directions \emph{on} the $Z$. This is easily remedied by a soft approximation argument. Indeed as the conclusion of Theorem~\ref{t:sharp3d} is an $L^2(\R^3)$-operator norm bound, we can work with functions $f\in L^2(\R^3)$ such that $\mathrm{supp}(\hat f)\subseteq B_3(R)$ for some $R>0$, as long as we prove bounds independent of $R$. Now if we choose $\delta\ll R^{-1}$ in Proposition~\ref{prp:pp}, we will have that $\|A_{\sigma}f\|_{L^2(\R^3)}\lesssim \| \mathrm{M}_{\tau(\sigma)}f\|_{L^2(\R^3)}$ for all $\sigma\in\Sigma_\times$, where $\tau(\sigma)$ denotes a direction on $Z$ such that $\mathrm{dist}(\sigma,\tau(\sigma))<\delta$, whose existence is guaranteed from point 1. of Proposition~\ref{prp:pp}. This remark allows us to assume that $\Sigma_\times\subset Z=\cup_j W_j$.

With the remark above taken as understood, the almost orthogonality principle of Theorem~\ref{t:sharp3d} reduces the proof of Theorem~\ref{t:sharp3d} to a recursive \emph{cellular} term, and a \emph{wall} term which is the $L^2(\R^3)$-norm of $A_{\Sigma_\times}$. The latter operator is a single scale, maximal directional average along directions on a one-dimensional algebraic subvariety of $\mathbb S^2$. It follows from \cite{DPPalg}*{Theorem D} that
\begin{equation}\label{eq:1dvariety}
\sup_{\substack{ \Sigma' \subset \cup_j W_j \\ \#\Sigma' \leq N }}\|A_{\Sigma' } \|_{L^2(\R^3)} \lesssim D^{\frac12}(\log N)^{\frac32}.	
\end{equation}

\subsection{The proof of Theorem~\ref{t:sharp3d}} We are seeking an estimate for 
\[
K_N \coloneqq \sup_{\substack{\Sigma\subset \mathbb S^2 \\ \#\Sigma\leq N }} \| A_{\Sigma}\|_{L^2(\R^n)}^2.
\]
Combining Theorem~\ref{t:alldortho} with the polynomial partition of Proposition~\ref{prp:pp}, the subsequent remarks, and \eqref{eq:1dvariety} we can estimate for any $D\gtrsim 1$
\[
K_N \leq  K_1 D (\log N)^{ 3} + K_2  D K_{\frac{N}{D^2}}
\]
for numerical constants $K_1,K_2>0$ and $D$ a sufficiently large degree, to be chosen momentarily. Indeed choosing $D\coloneqq \sqrt{N}/(\log N)^3 $ yields
\[
\frac{K_N}{\sqrt{N}} \leq K_1 +K_2 \frac{K_{(\log N)^6}}{\sqrt{(\log N)^{6}}}
\]
which readily implies the estimate in statement of Theorem~\ref{t:sharp3d} by recursion.

\section{Kakeya and Nikodym maximal operators} \label{s:kakeya} In this section we digress a bit  in order to discuss two maximal operators that appear naturally in the context of this paper, namely the \emph{Kakeya} and \emph{Nikodym} maximal operators. These operators have been studied extensively in the case $d=1$ in relation to the maximal Kakeya conjecture and the maximal Nikodym conjecture, which are equivalent when $d=1$; see \cite{Tao}. For $d>1$, the  corresponding Kakeya maximal function on $\mathrm Gr(d,n)$ has been studied in relation to the existence and dimension of $(d,n)$-Besicovitch sets.  We send the interested reader to \cite{Mattila} for general background on the topic, and will focus below  on just a few  notions central to our discussion.

We briefly recall some elementary properties of the Haar measure on $\mathrm Gr(d,n)$. Let $1\leq d<n$ be fixed parameters and  $\mathrm{d}\sigma $ be the quotient Haar measure on $\mathrm{Gr}(d,n)$ seen as a quotient group as in \eqref{e:qgroup}, that is
\[
\int\displaylimits_{\mathrm{Gr}(d,n)} f(\sigma)\,  \d\sigma = \int\displaylimits_{O(n)} f(\omega) \, \d \omega 
\] 
where $O(n)$ is equipped with its left-invariant Haar probability measure. When $A\subset  \mathrm{Gr}(d,n)$, the notation $|A|$ stands for the $\d\sigma$-measure of $A$.
As the measure  $\mathrm{d}\sigma $ is the unique  probability measure on $ \mathrm{Gr}(d,n)$ which is left invariant under the action of $O(n)$, it coincides with the normalized Riemannian volume form on $ \mathrm{Gr}(d,n)$. This implies that, if 
\[
{\mathsf B}_{\delta}(\sigma)\coloneqq \{\tau \in \mathrm{Gr}(d,n):\,\mathsf{d}(\sigma,\tau)<\delta\},
\]  
the measure $|{\mathsf B}_{\delta}(\sigma)|$ is independent of $\sigma$ and $|{\mathsf B}_{\delta}(\sigma)|  \sim \delta^{d(n-d)}$.

\subsection{The Kakeya maximal operator} Recall from \eqref{e:plate} the notation $T_\delta(\sigma)$ for the $\delta$-plate oriented along $\sigma \in \mathrm{Gr}(d,n)$
Consider the dual maximal operator acting on $f\in L^1 _{\mathrm{loc}}(\R^n)$  and its tailed version
 \[
 \K_{\delta} f(\sigma) = \sup_{x\in \R^n} \fint\displaylimits_{x+ T_\delta(\sigma)} |f|, \qquad  \widetilde {\K_{\delta}} f(\sigma) = \sup_{x\in \R^n} \sum_{k=0}^\infty 2^{-2kn} \fint\displaylimits_{x+ 2^{k}T_\delta(\sigma)} |f|,\qquad \sigma\in \mathrm{Gr}(d,n).
 \]

\begin{definition} Let $1\leq d <n$. A Borel set $A\subset \R^n$ is said to be a \emph{$(d,n)$-set} if for every $\sigma \in \mathrm{Gr}(d,n)$ there is $y\in\R^n$ such that $B_n(y,1)\cap (\sigma+y)\subset E$. If $|A|=0$ then $A$ is called a $\emph{$(d,n)$-Besicovitch}$ set.
\end{definition}
It is well known that for $d=1$ there exist $(1,n)$-Besicovitch or \emph{Kakeya} sets and the \emph{Kakeya conjecture} states that the Hausdorff dimension of Kakeya sets should be at least $n$. The \emph{maximal Kakeya} conjecture is the statement that for all $\delta,\eps>0$ we should have the estimate
\begin{equation}\label{eq:maxkak}
\|\K_\delta\|_{L^n(\mathrm{Gr}(1,n))}\lesssim_\eps \delta^{-\eps}\|f\|_{L^n(\R^n)}.
\end{equation}
When $1<d<n$, it is conjectured that no $(d,n)$-Besicovitch sets exist. For $d>n/2$  this conjecture was verified by Falconer, \cite{Falc}. The range of non-existence of $(d,n)$-Besicovitch sets was extended by Bourgain in \cite{bourgain_gafa} and further by Oberlin in \cite{Oberlin_thesis}. The reader is also referred to  \cite{Oberlin_thesis} for further results and lower bounds on the Hausdorff dimensions of $(d,n)$-Besicovitch for general $d$. The connection with the Kakeya maximal function is revealed by the following well known implication; see \cite[p. 3]{Oberlin_thesis}.

\begin{proposition}\label{prop:kakeya} Let $1\leq d<n$ and suppose that there exists  $1\leq p <\infty$ and $\eps>0$ such that for every $\delta>0$ we have the bound
	\[
	\|\K_\delta f\|_{L^p(\mathrm{Gr}(d,n))} \lesssim_\eps \delta ^{-\frac{\eps}{p}}\|f\|_{L^p(B_n)}.
	\] 
Then every $(d,n)$-set has Hausdorff dimension at least $n-\eps$. 	
\end{proposition}

In the context of the current paper the case $1<d<n$ is the most relevant. The estimate of the following proposition is a natural consequence of the methods of this paper and we include as an illustration of how these methods can be applied in the context of the $(d,n)$-Kakeya maximal function. Our estimate below recovers the well known result of Falconer \cite{Falc}:   there are no $(d,n)$-Besicovitch sets when $d>n/2$.  By Proposition~\ref{prop:kakeya}, it also yields that a $(d,n)$-Besicovitch set necessarily has full Hausdorff dimension  when $d=n/2$.

\begin{proposition}$
\displaystyle
\left\|  \K_{\delta} :L^2(\R^n) \to L^2(\mathrm{Gr}(d,n))   \right\|  \lesssim 
\begin{cases} 
1 & 2d> n 
\\ 
\log \delta & 2d=n 
\\
\delta^{d-\frac{n}{2}} & 2d< n. 
\end{cases} 
$
\end{proposition}

\begin{proof}  By an isotropic rescaling of the input function $f$,  the norm equivalence
\begin{equation}
\label{e:tricky0}
\left\|  \K_{\delta} :L^p(\R^n) \to L^q(\mathrm{Gr}(d,n))   \right\| \sim_{p,q}
\left\| \widetilde{ \K_{\delta} }:L^p(\R^n) \to L^q(\mathrm{Gr}(d,n))   \right\|  
\qquad \forall 1\leq p,q \leq \infty
\end{equation} holds.
By standard arguments and Lemma \ref{l:1}, if  $ \tau \in \mathrm{Gr}(d,n) $ and $\sigma \in B_{2^{-10}\delta}(\tau)$ we have the pointwise estimates
\begin{equation}
\label{e:tricky1}
 \K_{\delta} f(\sigma) \lesssim   \sup_{x\in \R^n} \left| A_{\sigma, 1}^{>\delta} f(x)\right| \lesssim   \sup_{x\in \R^n} \mathsf{A}^\delta_\tau f(x) \lesssim    \widetilde {\K_{\delta}} f(\tau) \lesssim  \delta^{-\frac{(n-d)}{p}} \|f\|_p.
\end{equation}
The proof of the Proposition relies on an inductive estimate for 
$
Q_\delta \coloneqq \left\|  \mathsf{K}_{\delta} :L^2(\R^n) \to L^2(\mathrm{Gr}(d,n))   \right\|,
$ having defined the
the Fourier localized operator
\[
\mathsf{K}_{\delta} f(\sigma) \coloneqq  \sup_{x\in \R^n} \left| A_{\sigma, 1}^{>\delta} f(x)\right| , \qquad  \sigma\in \mathrm{Gr}(d,n).
\]
Let $\Phi$ be as in \eqref{e:LPsplit} and set $f_\delta\coloneqq  \mathcal{F}^{-1} [ \widehat f (\cdot)(1- \Phi(2^{4} \delta \cdot))]. $
 Then \begin{equation}
A_{\sigma, 1}^{>\delta} f = A_{\sigma, 1}^{>4\delta} f + A_{\sigma, 1}^{>\delta} f_\delta, \qquad \sigma \in \mathrm{Gr}(d,n)
 \end{equation}
which tells us immediately that
\begin{equation} \label{e:inductive0}
Q_{\delta} \leq Q_{4\delta} + \sup_{\|f\|_2=1} \left\| \mathsf{K}_{\delta} f_\delta\right\|_2.
 \end{equation}
 The important fact to notice here is that due to the support of $f_\delta$ being contained in $\mathrm{Ann}^+(\delta)$, we have   
 \[
 \sigma \in B_{h\delta}(\tau) \implies
 A_{\sigma, 1}^{>\delta} f_\delta = A_{\sigma, 1}^{>\delta} O_{\tau,\delta} f 
 \]
 where $O_{\tau,\delta}$ is the frequency cutoff to the set  $\ O_{\tau,\delta}\coloneqq\{\xi \in \mathrm{Ann}^+(\delta):\, |\pr_\tau \xi|\leq 2\delta |\xi|\}$. Notice that by Lemma \ref{l:grasm} we have for every $\delta$-separated set $\Sigma_\delta \subset \Grdn$ that
 \begin{equation}
\label{e:deltaortho2}
 \sum_{\tau \in   \Sigma_\delta} \|O_{\tau,\delta} f\|_2^2 \lesssim \delta^{-d(n-d-1)} \|  f\|_2^2
 \end{equation}
Therefore if $\Sigma_\delta$ is a maximal $2^{-11} \delta$-net  in $ \mathrm{Gr}(d,n)$, relying on \eqref{e:tricky1} to pass to the second line and on the overlap estimate  \eqref{e:deltaortho2}
\begin{equation}
\label{e:inductive1}
\begin{split}
\left\| \mathsf{K}_{\delta} f_\delta\right\|_2^2  &\leq  \delta^{d(n-d)} \sum_{\tau \in \Sigma_\delta}   \sup_{\sigma \in {\mathsf B}_{2^{-10}\delta}(\tau)  }
\sup_{x\in \R^n} \left| A_{\sigma, 1}^{>\delta} O_{\tau,\delta} f(x)\right|^2
\\ & \lesssim  \delta^{(d-1)(n-d)} \sum_{\tau \in \Sigma_\delta}   \|O_{\tau,\delta} f\|_2^2 \lesssim \delta^{2d-n}
\end{split}
\end{equation}
Combining \eqref{e:inductive1} with \eqref{e:inductive0} yields the recursion 
\[
Q_{\delta} \leq Q_{4\delta} + \Theta \delta^{2d-n}. 
\]
for some dimensional constant $\Theta$. This proves the proposition via easy induction.
\end{proof}

\subsection{The Nikodym maximal operator} For $1\leq d<n$, $f\in L^1 _{\mathrm{loc}}(\R^n)$, consider the maximal $\delta$-plate averages
\[
\mathcal N_\delta f(x)\coloneqq \sup_{ \sigma\in\mathrm{Gr}(d,n)} \fint_{x+T_\delta(\sigma)}|f|,\qquad x\in\R^n.
\]
The role of $\Grdn$  is kept implicit in the notation of the Nikodym maximal function $\mathcal N_\delta$. The study of the Nikodym maximal operator is motivated by the question of existence and dimension of $(d,n)$-Nikodym sets as defined below.

\begin{definition} We say that $A\subset \R^n$ is a $(d,n)$-Nikodym set if $|A|=0$ and for every $x\in A$ there exists an affine $d$-plane $\sigma+y$, with $\sigma \in \mathrm{Gr}(d,n)$ and $y\in\R^n$, such that
\[
 x\in \sigma+y \quad\text{and}\quad	  B_n(y,1) \cap (\sigma+y)  \subset E.
\]
\end{definition}
For $d=1$ it is easy to see that the Kakeya conjecture would imply that all $(1,n)$-Nikodym sets have Hausdorff dimension at least $n$. Again for the case $d=1$ the maximal Kakeya conjecture \eqref{eq:maxkak} is equivalent to the statement that for all $\delta,\eps>0$ we have
\[
\|\mathcal N_\delta f\|_{L^n(\R^n)}\lesssim_\eps \delta^{-\eps}\|f\|_{L^n(\R^n)},\qquad d=1.
\]
On the other hand, while for $d>n/2$ there are no $(d,n)$-Besicovitch sets, it is known that for all $1\leq d <n$ there exist $(d,n)$-Nikodym sets; see \cite[Corollary 6.6]{FalcPLMS}. It is thus natural to focus on lower bounds for the Hausdorff dimension of such Nikodym sets. In order to formulate the relevant conjecture on the maximal level we briefly explore below the relation of Hausdorff dimension with $L^p(\R^n)$-bounds for $\mathcal N_\delta$. The proof is similar to \cite[Theorem 22.9]{Mattila}.

\begin{proposition} Suppose that there exists $1\leq p<\infty$, $1\leq q<\infty$  and $\eps>0$ such that the following estimate holds. For all $\delta,\eps>0$ there holds
\[
\| \Nn_\delta f :\, L^{p}(\R^n )\to L^{p}(\R^n) \| \lesssim_\eps \delta^{-\frac{\eps}{p}}.
\]
Then every $(d,n)$-Nikodym set has Hausdorff dimension at least $n-\eps$.
\end{proposition}

%

For general $1\leq d<n$ the counterexample presented in Proposition~\ref{prop:cex} for the maximal operator $\M_{\Sigma,\{1\}}$ can be easily modified for the Nikodym maximal function with $\delta^{-1}=N^{1/(n-d)}$, showing that the following conjecture would yield the best possible quantification of the range of boundedness.

\begin{conjecture}\label{conj:nik} Let $\mathcal N_\delta$ denote the $(d,n)$-Nikodym maximal function for $1\leq d< n$. There holds
	\[
	\| \Nn_\delta  :\, L^{p}(\R^n)\to L^{p}(\R^n) \| \lesssim 
	\begin{cases} \delta^{-\frac{n-d+1-p}{p}},&\qquad 1<p<n-d+1,
		\\
		(\log \delta^{-1})^{\frac{1}{p}},&\qquad p\geq n-d+1.
	\end{cases}
	\]
\end{conjecture}

The conjecture above suggests that the critical exponent for the $(d,n)$-Nikodym problem is $p_o\coloneqq n-d+1$. We verify the conjecture in the case $d=n-1$, so that $p_o=2$, in Section~\ref{s:cod1} below; in fact we prove the same logarithmic dependence for a more general operator given with respect to $(n-1)$-plates in $\R^n$ of arbitrary eccentricity. A range of off-diagonal estimates can be conjectured by interpolating the estimate of Conjecture~\ref{conj:nik} with the trivial $L^1\to L^\infty$ estimate.

\begin{conjecture} Let $\mathcal N_\delta$ denote the $(d,n)$-Nikodym maximal function for $1\leq d< n$. For $1\leq p \leq n-d+1$, $q=p'(n-d)$, and all $\eps>0$
	\[
	\| \Nn_\delta f :\, L^{p}(\R^n)\to L^{q}(\R^n) \| \lesssim_\eps  \delta^{-\big(\frac{n-d+1-p}{p}\big)} (\log \delta^{-1})^{\frac{1}{p'(n-d)}}
	\]
\end{conjecture}

For the case of general $1<d<n$, we include here the proof of the best possible $L^2(\R^n)$-estimate for the $(d,n)$-Nikodym maximal operator, namely the proof of Theorem~\ref{thm:L2nik} as stated in the introduction of the paper. Unlike the previous parts of this paper, the proof of Theorem~\ref{thm:L2nik}, as well as the one yielding the sharp critical bound for $d=n-1$ in Section~\ref{s:cod1} below, rely on geometric rather than Fourier analytic considerations, together with $TT^*$-types of arguments. In particular we will rely on precise estimate for pairwise intersections of $\delta$-neighborhoods of elements of $\Grdn$. When $n\geq 2d$ pairwise intersections of elements of $\Grdn$ may have any dimension between $0$ and $d-1$, requiring a corresponding classification of the volume estimates; these estimates are stated and proved in Lemma~\ref{l:inters} below. Similar estimates have appeared in \cite{Rog} for $d=2$.

\subsection{Pairwise intersections in $\Grdn$ and volume estimates} 
For $\sigma \in \mathrm{Gr}(d,n)$ it is convenient to denote  $\mathbb S^{\sigma}\coloneqq \mathbb S^{n} \cap \sigma$. 
Let $\sigma,\tau\in \mathrm{Gr}(d,n)$ and suppose that  $\eta=\sigma\cap \tau \in \mathrm{Gr}(m,n)$ with $0\leq m <d$. We define the   angles \[0<\theta_{m+1}(\sigma, \tau)\leq \cdots \leq \theta_{d}(\sigma, \tau)\]
 inductively as follows. 
Initially set  $\sigma_{m+1}= \sigma \cap \eta^\perp$, $\tau_{m+1}=\tau\cap \eta^\perp$. For $m+1\leq j<d$ and $\sigma_{j}, \tau_{j}$ have been defined let
\[
\begin{split}
&(s_j,t_j) \coloneqq \arg \min \left\{\arccos ( s t) :\,\, s\in   \mathbb{S}^{\sigma_{j}}, \, t\in    \mathbb{S}^{\tau_{j}}\right\}, \qquad\theta_j  \coloneqq \arccos ( s_j t_j), \\ & \sigma_{j+1}\coloneqq \sigma_{j} \cap s_j^\perp, \qquad \tau_{j+1}\coloneqq \tau_{j} \cap t_j^\perp,
\end{split}\]
and repeat with $j+1$ in place of $j$. The algorithm stops when $j=d$.

\begin{remark}Notice that $\theta_{m+1}>0$, as there is no pair $(s,t)\in \mathbb{S}^{\sigma_{m+1}}\times \mathbb{S}^{ \tau_{m+1}} $ with $s$ and $t$ colinear, and that clearly $\theta_{j}$ is nondecreasing.
Also notice that if $s_1=t_1, \cdots s_{m}=t_m$ is an orthonormal basis of $\sigma\cap \tau$, then $\{s_1,\ldots, s_d\}$ and  $\{t_1,\ldots, t_d\}$ are   orthonormal bases of $\sigma$, $\tau$ respectively. { In addition, the construction of $s_{m+1},\ldots, s_{d}$, $t_{m+1},\ldots, t_{d}$ yields a further orthogonality property; namely if $\pi_j=\mathrm{span}\{s_j,t_j\}$ for $m+1\leq j \leq d$ then 
\[
1\leq j<k \leq d\implies
xy=0 \qquad \forall x\in \pi_j, y\in \pi_k.
\]
This is obvious if $j\leq m$. Otherwise,   if  $t\in \tau_k$ is such that $s_jt\neq 0$ for some $j<k$, we may write $t=(\cos\phi) \widetilde t + (\sin \phi) s_{j}$ for some $\phi\neq 0$ and $\widetilde t\in \tau_k$. Thus for any $s\in \sigma_k$ as $ss_j=0$ we have
\[
st =(\cos\phi) s\widetilde t + (\sin \phi) ss_{j}= (\cos\phi) s\widetilde t  < s \widetilde t 
\]
thus no  $s\in \sigma_k$ exists such that $(s,t)$ is a minimizer.  It follows that $t_{k}$ is orthogonal to $s_j$ whence $xy=0   $ for all $x\in \pi_j, y\in \pi_k.$ The angles and bases
\[
0=\theta_1=\cdots=\theta_m,\quad 0<\theta_{m+1}\leq \cdots \leq  \theta_{d}\leq \frac{\pi}{2}, \qquad S=\{s_1,\ldots, s_d\}, \qquad  T=\{t_1,\ldots, t_d\},
\] 
are respectively  called  \emph{canonical angles and bases} and may be obtained, respectively, as inverse cosines of the eigenvalues and eigenvectors of the $2d\times 2d$ matrix $M$  of inner products between elements of arbitrary orthonormal bases of $\sigma$ and $\tau$. 
}
\label{rem:angle}
Finally we stress that $\theta_{d}(\sigma,\tau)\sim\mathsf{d}(\sigma,\tau)$, where the latter refers to the $\mathrm{Gr}(d,n)$ distance. For completeness, if $m>0$, we may set $\theta_j(\sigma,\tau)=0$ for $1\leq j\leq m$.
\end{remark}

\begin{remark} \label{rem:angles2} Given $\sigma,\tau\in \mathrm{Gr}(d,n)$, suppose  $\sigma\cap\tau\in \mathrm{Gr}(m,n)$, so that $\zeta=\mathrm{span}\,\{\sigma,\tau\}\in \mathrm{Gr}(2d-m,n)$. Let $S=\{s_1,\ldots, s_d\}$ and $T=\{t_1,\ldots, t_d\}$ be the principal bases of $\sigma,\tau$ respectively that we constructed above. For each $j=m+1,\ldots,d $ let $z_j$ be the unit vector so that \[
t_j=(\cos \theta_j) s_j +(\sin \theta_j)z_j.\]  Notice  that $z_j$ belongs to $\pi_j$ and is thus orthogonal to $s_k$ for all $k\neq j$ by Remark  \ref{rem:angle} and to $s_j$ by construction. It follows that
\[
Z=\{s_1,\ldots, s_d, z_{m+1},\ldots, z_d\}
\]
is an orthonormal basis of $\zeta$. 
\end{remark}

\begin{lemma} \label{l:inters}
Let $\sigma,\tau\in \mathrm{Gr}(d,n)$ and suppose that $\eta=\sigma\cap \tau \in \mathrm{Gr}(m,n)$. Let $a,b\in \R^n,$ $P(\sigma)=a+T_\delta(\sigma), P(\tau)=b+T_\delta(\tau)$ be $\delta$-plates. Then
\[
\left|P(\sigma)\cap P(\tau)\right| \lesssim_{n,d} \delta^{n-m}\left(  \prod_{j=m+1}^d \max\{  \delta, \theta_{j}(\sigma, \tau) \}\right)^{-1}.
\]
\end{lemma}

\begin{proof} Let $\{s_1,\ldots, s_d\}$ and  $\{t_1,\ldots, t_d\}$ be the   orthonormal bases of $\sigma$, $\tau$ respectively we obtained with the principal angle construction. Let $\{s_{d+1}, \ldots s_{n}\}$ also be a basis of $\sigma^\perp$. Pick any point $p\in P(\sigma)\cap P(\tau)$. Choose coordinates $y=(y_1,\ldots, y_j)\in \R^n$ so that $p$ is the origin and $y_j=y\cdot s_j$.  We claim that $P(\sigma)\cap P(\tau)$ is contained in the intersection of the $n$ bands
\[
B_{j} \coloneqq  
\begin{cases} 
\left\{y\in \R^n : |y_j| <  3\right\}, & j=1,\ldots, m,\vspace{.1em}
\\
\left\{y\in \R^n : |y_j| < \frac{C\delta}{\max\{\delta,\theta_j\}}\right\}, & j=m+1,\ldots, d,\vspace{.1em}
\\
 \left\{y\in \R^n : |y_j| < 3\delta \right\} , & j=d+1,\ldots, n.
\end{cases}   
\]
The claim readily yields the conclusion of the lemma and is
 is completely  obvious for $j=1,\ldots, m$ and $j=d+1, \ldots, n$, because for those values of $j$ one has $P(\sigma)\subset B_{j}$ as well.
It is also obvious for $j\in\{m+1,\ldots, d\}$ if $\max\{\delta,\theta_j\}=\delta$, thus we fix now $j\in\{m+1,\ldots, d\}$ and  arguing by contradiction, suppose that there exists $y\in P(\sigma)\cap P(\tau)$ with $|y_j|\geq \frac{C\delta}{\theta_j}$.  Let $z_j$ be as in Remark \ref{rem:angles2}. It follows that 
{
\[
|y\cdot z_j |\leq  2\delta.
\]
}However simple geometry shows that
\[
\tan\theta_j = \frac{|y\cdot z_j |}{|y_j|} \leq \frac{2\delta}{\frac{C\delta}{\theta_j}} = \frac{2\theta_j}{C},
\]
a contradiction if $C$ is large enough.
\end{proof}
In the next lemma we describe how to construct the essentially minimal dilate of a $\delta$-plate $P(\sigma)$ that covers a nearby plate $P(\tau)$, depending on the principal angles of $\sigma, \tau.$  We need some notation first

Let $\sigma,\tau\in \mathrm{Gr}(d,n)$ and suppose that $\sigma\cap \tau \in \mathrm{Gr}(m,n)$. Let $S,T$ and $Z$ be the respective orthonormal bases of $\sigma,\tau$ and $\zeta=\mathrm{span}(\sigma,\tau)$ constructed in Remark \ref{rem:angles2}. Then the plates $T_\delta(\sigma), T_\delta(\tau) $ may be described by
\[
\begin{split}&
T_{\delta}(\sigma)=\left\{x :\,\max_{\{1,\ldots, d\}} |x\cdot s_j| <1, \, \max_{m+1\leq j\leq d} |x\cdot z_j|<\delta,\, \max_{w\in \zeta^\perp} |x\cdot w|<\delta \right\},
\\ &
T_{\delta}(\tau)=\left\{x :\,\max_{\{1,\ldots, m\}} |x\cdot s_j| <1, \, \max_{\{m+1,\ldots, d\}} |x\cdot[(\cos \theta_j) s_j + (\sin\theta_j) z_j ]  |<1, \,  \max_{w\in \tau^\perp} |x\cdot w|<\delta \right\}.
\end{split}
\]
We define the dilation $T_{\delta}^{+\tau}(\sigma)$ of $T_{\delta}(\sigma)$ by
\[
T_{\delta}^{+\tau}(\sigma)\coloneqq \left\{x :\, \max_{\{1,\ldots, d\}} |x\cdot s_j| <3, \, \max_{m+1\leq j\leq d} |x\cdot z_j|<3\max\{\delta,\theta_j\},\, \max_{w\in \zeta^\perp} |x\cdot w|<3\delta \right\}.
\]

\begin{lemma}  \label{lem:inclT+}
Let $a,b\in \R^n,$ $P(\sigma)=a+T_\delta(\sigma), P(\tau)=b+T_\delta(\tau)$ be $\delta$-plates with $P(\sigma)\cap P(\tau)\neq \varnothing$. Then
\[
P(\tau)\subset a+ T_{\delta}^+\tau(\sigma).
\] 
\end{lemma}

\begin{proof} By translation invariance, we may assume  $a=0$. Then $P(\tau)$ is contained in the moderate dilate
\[
3T_{\delta}(\tau)\coloneqq\left\{x :\max_{\{1,\ldots, m\}} |x\cdot s_j| <3, \, \max_{\{m+1,\ldots, d\}} |x\cdot[(\cos \theta_j) s_j + (\sin\theta_j) z_j ]  |<3, \,  \max_{w\in \tau^\perp} |x\cdot w|<3\delta \right\}.
\]
Simple geometry tells us that when $j\in \{m+1,\ldots, d\} $
\[
x\in 3T_{\delta}(\tau)\implies |x\cdot[(\cos \theta_j) s_j + (\sin\theta_j) z_j|\leq 3 \implies |x\cdot z_j| \leq 3\sin \theta_j\leq 3\theta_j.
\]
On the other hand it is obvious that $|x\cdot s_j|<3$ for all $1\leq j\leq d$ when $x\in 3T_{\delta}(\tau)$. Therefore $3T_{\delta}(\tau)\subset   T_{\delta}^{+\tau}(\sigma)$, and the proof is complete.
\end{proof}

\subsection{The proof of Theorem~\ref{thm:L2nik}} As noted in the discussion leading to the formulation of Conjecture~\ref{conj:nik}, the bound of  Theorem~\ref{thm:L2nik} is best possible so it suffices to prove the upper bound. We begin by linearizing the maximal operator $\mathcal N_\delta$ as follows. For $\Sigma\subset \Grdn$ that will remain fixed throughout the proof we let $\mathcal T_\Sigma$ denote the collection of all $\delta$-plates of the form $a_T+T_\delta(\sigma_T)$ for $a_T\in\R^n$ and $\sigma_T\in\Sigma$.  Now given $f\in \mathcal S(\R^n)$ and $\mathcal T\subset\mathcal T_\Sigma$ we consider the linear operator
\[
\mathcal N_{\mathcal T} f (x)\coloneqq \sum_{T\in\mathcal T} \bigg(\fint_T f \bigg) \ind_{F_T}(x)
\]
where $F_T\subset T$ for every $T\in\mathcal T$ and the collection $\{F_T\}_{T\in\mathcal T}$ is pairwise disjoint. Denoting by $\mathcal N_T ^*$ the adjoint of $\mathcal N_T$ we have that
\[
\|\mathcal N_{\delta}: \, L^2(\R^n)\to L^{2,\infty}(\R^n)\|\eqsim \sup_{\mathcal T\subset \mathcal T_\Sigma} \|\mathcal N_{\mathcal T}: \, L^2(\R^n)\to L^{2,\infty}(\R^n)\|\eqsim\sup_{\mathcal T\subset \mathcal T_\Sigma} \sup_{E\subset \R^n} \frac{\|\mathcal N_{\mathcal T} ^* \ind_E\|_2}{|E|^{\frac12}},
\]
where the supremum in ${\mathcal T\subset \mathcal T_\Sigma}$ can be taken over finite collections $\mathcal T$.
As the adjoint operator has the form
\[
\mathcal N_{\mathcal T} ^* g \coloneqq \sum_{T\in{\mathcal T} } \bigg(\frac{1}{|T|}\int_{F_T} g\bigg) \ind_T(x)
\]
we readily see that
\[
\mathcal N_T ^* \ind_E =\sum_{T\in{\mathcal T} } |F_T\cap E| \frac{\ind_T}{|T|}\eqqcolon \sum_{T\in{\mathcal T} }|E_T|\frac{\ind_T}{|T|}
\]
with $E_T\coloneqq F_T \cap E$ pairwise disjoint and $E_T\subset T$. Note that the collection $\{|E_T|\}_{T\in\mathcal T}$ is a Carleson sequence: for any open set $U\subset\R^n$ we have
\[
\sum_{\substack{T\in\mathcal T\\ T\subset U}}|E_T|\leq |U\cap E|.
\]

We now expand the square of the $L^2$-norm as follows
\[
\| \mathcal N_{\mathcal T  } ^* \ind_E\|_2 ^2 =\sum_{T\in{\mathcal T} } |E_T| \sum_{\substack{T' \in{\mathcal T} \\ T\cap T' \neq \varnothing}} |E_{T'}|\frac{|T\cap T'|}{|T||T'|}.
\]
For fixed $T\in{\mathcal T}$ and $0 \leq k \leq \log \delta^{-1}$ we define
\[
\mathcal T_{k}(T)\coloneqq \{ T' \in\mathcal T:\,\, T\cap T' \neq \varnothing, \, 2^{k}\delta\leq \dist(T,T')\leq 2^{k+1}\delta \}. 
\]
This notation allows us to write
\begin{equation}\label{eq:2^kdecomp}
\| \mathcal N_{\mathcal T} ^* \ind_E\|_2 ^2 =\sum_{k=0} ^{\log \delta^{-1}} \sum_{T\in\mathcal T}\frac{|E_T|}{|T|} \sum_{T'\in \mathcal T_{k}(T)} |E_{T'}|\frac{|T\cap T'|}{|T'|}.
\end{equation}
Now for fixed $k,T$ and $T'\in \mathcal T_k(T)$ we use Lemma~\ref{l:inters} to estimate for every $m\in \{0,\ldots,d-1\}$
\[
\frac{|T\cap T'|}{|T'|} \leq \delta^{n-m} \frac{1}{\delta^{d-m-1} \dist(\sigma_T,\sigma_{T'})}\frac{1}{|T'|}\eqsim \frac{\delta^{n-d+1}}{2^k\delta}\frac{1}{\delta^{n-d}}=2^{-k}.
\]
Then for every fixed $k$ we will have by Lemma~\ref{lem:inclT+} and the Carleson property of the sequence $\{|E_T|\}_{T\in\mathcal T}$
\[
\sum_{T' \in \mathcal T_k }|E_{T'}|\leq \bigg| \bigcup_{T' \in \mathcal T_k } T' \bigg| \leq |T_\delta ^+(\sigma_T)|.
\]
Since $\theta_j(\sigma_T,\sigma_{T'})\leq \theta_d(\sigma_T,\sigma_{T'})\eqsim 2^k\delta$ and $2^k \delta \geq \delta$ we can use the definition of $T^+ _\delta (\sigma_T)$ to estimate
\[
|T_\delta ^+ (\sigma_T) |\lesssim (2^k\delta)^{n-d}.
\]
Using the estimate in the last two displays and the calculation in \eqref{eq:2^kdecomp} we gather
\[
\| \mathcal N_{\mathcal T} ^* \ind_E\|_2 ^2 \lesssim \sum_{k=0} ^{\log \delta^{-1}} \sum_{T\in\mathcal T}\frac{|E_T|}{|T|} (2^k\delta)^{n-d} 2^{-k} \lesssim \begin{cases}
\delta^{-(n-d-1)}|E|,\qquad 1\leq d<n-1,\vspace{.4em}
\\
\log(\delta^{-1})|E|,\qquad d=n-1,
\end{cases}
\]
which proves the desired weak $(2,2)$ norm-estimate. The corresponding strong $(2,2)$ bound with an additional $\sqrt{\log\delta^{-1}}$-term follows from the corresponding weak-type $(2,2)$ bound and a well known interpolation argument of Str\"omberg, \cite{Stromberg}; see also \cite{Katz}*{p. 77--78} for the details of this argument.

\section{The sharp bound for codimension one maximal operators} \label{s:cod1} 
In this section, we prove Theorem \ref{thm:codim1_intro} as a consequence of a more general directional Carleson embedding theorem, Theorem~\ref{thm:carleson} below. Therefore,  we work with fixed codimension  $1$, so that $n=d+1\geq 2$ throughout. Recall that our goal is to prove a sharp estimate in terms of the cardinality parameter $N$ for  the maximal $d$-subspace averaging operator $\mathrm{M}_{\Sigma,(0,\infty)} $ when $\Sigma\subset \mathrm{Gr}(d,d+1)$ is a finite set with $\#\Sigma=N$. 

Our setup is more conveniently described by taking advantage of the isometric isomorphism 
\[
\sigma \in \Sigma \subset \mathrm{Gr}(d,d+1) \mapsto v=\sigma^\perp\subset  \mathrm{Gr}(1,d+1) \sim \mathbb S^d
\]
 By finite splitting and rotational invariance, we may work under the assumption that   $V\subset \mathbb S^{d}$ is   contained in a small neighborhood of $e_n$. This choice of coordinate system is conveniently exploited by modifying slightly our definition of $d$-plate \eqref{e:plate}, as follows.   For a $d$-dimensional {axis-parallel} cube $I\subset e_n^\perp$  with center $(c_I,0)\in \R^n$ and sidelength $\ell_I$, for $v\in V$, an interval $K \subset \R$ with $|K|\leq  \ell_I$, let 
\[
\begin{split}
&p(I,t,v)\coloneqq \{y\in v^\perp + (c_I,t):\, \pr_{e_n^\perp} y \in I\},
 \qquad P(I,K,v) \coloneqq \bigcup_{t\in K}  p(I,t,v).
 \end{split}
\]
The set $Q=P(I,K,v)$ is a $d$-\emph{plate} with orientation $v_Q=v$, basis $I_Q=I$,  scale $s_Q=\ell_I$,  height $K_Q=K$ and center $c_Q=(c_{I_Q},c_{K_Q})$, where $c_{K_Q}$ is the center of the interval $K\subset \R$.  Each set $p(I_Q,t,v_Q)$ with $t\in K_Q$ is referred to as the $t$-\emph{slice} of $Q$.
\begin{remark} The plate $Q$ is  the  shearing of an axis-parallel box whose short side is oriented along $e_n$. To compare with \eqref{e:plate}, observe that if $\delta=\ell_{K_Q}/\ell_{I_Q}$, then $Q$ and $Q'\coloneqq c_Q+ s_QT_{\delta}(v_Q^\perp)$  are comparable, that is $Q\subset CQ',Q'\subset C Q$ for a suitably chosen dimensional constant $C$.
\end{remark}

We will work with different special collections of $d$-plates which we define below.
\begin{definition} Let  $V\subset \mathbb S^{d}$ be a finite set of directions, $\delta>0$ be a small parameter.
\begin{itemize}
\item[$\cdot$] The collection of all $d$-plates in $\R^n$ with orientation along $v\in V$ will be denoted by $\mathcal P_v$, and  $\mathcal P_{V }\coloneqq \bigcup_{v\in V}\mathcal P_v$. 
\item[$\cdot$] For $Q=P(I,K,v)\in\mathcal P_V$, write $\pl_Q\coloneqq v^\perp +c_Q$ and  call $\pl_Q$ the \emph{plane of Q}.
\item[$\cdot$] A $d$-plate $Q(I,K,v)$ will be called a $(d,\delta)$-plate if $\ell(K_Q)=\delta>0$, namely if it is a $d$-plate with fixed thickness $\delta>0$. The subcollection of those $(d,\delta)$-plates belonging to   $\mathcal P_v$ is referred to by $\mathcal P_{v,\delta}$ and $\mathcal P_{V,\delta}\coloneqq\cup_v \mathcal P_{v,\delta}$.
\item[$\cdot$] Given a dyadic grid $\mathcal D$ in $\R^{d}$,  special subcollections of $\mathcal P_v, \mathcal P_V$ are produced by defining $\mathcal D_v\coloneqq\{Q\in\mathcal P_v:\,I_Q\in\mathcal D\}$, and $\mathcal D_V\coloneqq \bigcup_{v\in V} \mathcal D_v$. The special subcollection of $\mathcal D_V$ consisting of $(d,\delta)$-plates will be denoted by $\mathcal D_{V,\delta}\coloneqq \bigcup_{v\in V}\mathcal D_{v,\delta}$. 
\item[$\cdot$]For a generic collection $\mathcal Q\subseteq \mathcal P_V$, set $\mathcal Q_v\coloneqq \{Q\in\mathcal Q:\, v_Q=v\}$. This yields  $\mathcal Q=\bigcup_{v\in V}\mathcal Q_v$.
\item[$\cdot$]A  partial order on $\mathcal D_V$ is defined as follows. If $Q,R\in\mathcal D_V$, say $Q\leq R$ if $Q\cap R\neq \varnothing$ and $I_Q\subseteq I_R$.
\item[$\cdot$] If $\mathcal Q\subset \mathcal P_V$ we will use the notation $\sh(\mathcal Q)\coloneqq \cup_{Q\in\mathcal Q} Q$ for the \emph{shadow} of the collection.
\end{itemize}
\end{definition}
With these definitions we introduce below maximal operators defined with respect to collection of plates. For any $\mathcal L\subseteq \mathcal P_V$, set
\[
\M_{\mathcal L} f(x)\coloneqq \sup_{\substack{L\in\mathcal L}}  \left( \, \avgint_{L} |f|\,\d y\right) \cic{1}_L(x),  \qquad x\in\R^n.
\]
Note that in general $\M_{\mathcal L}$ is a \emph{directional operator} as the plates have variable orientation. In the special case that $\mathcal L=\mathcal L_v \subseteq \mathcal P_v$ for some fixed $v\in V$ then $\M_{\mathcal L_v}$ is pointwise bounded by the strong maximal function in a suitable coordinate system. Another special case of interest occurs when $\mathcal L\subseteq \mathcal P_{v,\delta}$ for fixed $v\in V$ and $\delta>0$, in which case $\M_{\mathcal L}$ is a one-parameter operator and satisfies weak $(1,1)$ bounds uniformly in $v$ and $\delta$. Note that the weak $(1,1)$-bound persists if $\mathcal L\subset \mathcal P_v$ is a collection of fixed eccentricity: in that case the operators $\M_{\mathcal L}$ are again of weak-type $(1,1)$ uniformly in $v\in V$ and the eccentricity of the collection.

%
%
\subsection*{Carleson sequences}Directional Carleson sequences of positive numbers $\{a_Q\}_{Q\in \mathcal D_V}$ are introduced in the string of definitions that follow. 

\begin{definition} Let $\mathcal L\subset \mathcal P_V$ be a collection of $d$-plates and let $v\in V$ be a fixed direction. The collection $\mathcal L$ is \emph{subordinate to} $\mathcal T\subset\mathcal P_v$ if for every $L\in \mathcal L$ there exists $T\in\mathcal T$ such that $L\subseteq T$. 
\end{definition}
We stress that $\mathcal T\subset \mathcal P_v$ in the definition above \emph{only contains plates with fixed orientation} $v$. 

\begin{definition}\label{def:carleson} Let $a=\{a_Q\}_{Q\in\mathcal D_V}$ be a sequence of positive numbers. The sequence $a$ is an ($L^\infty$-normalized) \emph{Carleson sequence} if for every $\mathcal L\subset \mathcal D_V$ which is subordinate to some $\mathcal T\subset P_{v}$ for some fixed $v\in V$ we have
	\[
	\sum_{L\in\mathcal L} a_L\leq |\sh(\mathcal T)|, \qquad 	\mass_a(\mathcal Q) \coloneqq \sum_{Q\in\mathcal D_V} a_Q <\infty.
	\]
For $\mathcal Q\subset \mathcal D_{V}$ and a Carleson sequence $a=\{a_Q\}_{Q\in\mathcal D_V}$ define the \emph{balayage}
\begin{equation}\label{eq:TQ}
T_{\mathcal Q}(a)(x)\coloneqq \sum_{Q\in\mathcal Q} a_Q \frac{\ind_Q(x)}{|Q|},\qquad x\in\R^n.
\end{equation}
\end{definition}
It follows from the definition above that if $a$ is a Carleson sequence and $\mathcal T\subset \mathcal P_v$ for some fixed $v\in V$ then $\mathcal T$ is subordinate to itself and thus $\mass_a(\mathcal T)\leq |\sh(\mathcal T)|$.
\subsection*{An $L^2$-Carleson embedding theorem for $d$-plates} Here we describe and prove the main result of this section, a directional Carleson embedding theorem for $d$-plates in $\R^n$. In order to state it we also introduce for any $\mathcal Q\subseteq \mathcal P_V$ the notation
\[
\widehat {\mathcal Q} \coloneqq \bigcup_{1\leq s \leq 100n} \bigcup_{Q\in\mathcal Q} (1+s)Q.
\]

\begin{theorem}\label{thm:carleson} Let $V\subset\mathbb S^d$ be a finite set of directions and 	$\mathcal Q\subseteq \mathcal D_{V}$ be a collection of $d$-plates in $\R^n$. We assume that the operators $\{M_{\mathcal Q_v}:\, v\in V\}$ satisfy
	\[
	\sup_{v\in V} \big\| M_{\widehat{\mathcal Q}_v}:\, L^1(\R^n)\to L^{1,\infty}(\R^n)\|\lesssim_n 1.
	\]
If $a=\{a_Q\}_{Q\in\mathcal D_{V,\delta}}$ is a Carleson sequence then
\[
\| T_{\mathcal Q}(a)\|_{L^2(\R^n)}\lesssim_n (\log\#V)^{\frac12} \mass_a (\mathcal Q) ^{\frac12}
\]
with implicit constant depending only upon dimension.
\end{theorem}
The  proof of Theorem~\ref{thm:carleson} begins with some reductions that  simplify and highlight the main argument.
First, for any $\mathcal Q\subseteq \mathcal D_{V}$ we expand the square of the $L^2$-norm in the statement of the theorem as  
\begin{equation}
\label{eq:mainsplit}
\begin{split}
\frac{1}{2}\|T_{\mathcal Q}(a)\|_{L^2(\R^n)} ^2 &\leq  \sum_{R\in\mathcal Q} a_R  \avgint_R \sum_{\mathcal Q\ni Q\leq R} a_Q \frac{\ind_Q}{|Q|} 
  \\
  & \leq  \mu (\log\#V )\mass_a(\mathcal Q) +\mu(\#V) \sum_{k>\mu(\log\#V)} (k+1) \sup_{v\in V} \sum_{R\in\mathcal Q_{v,k} } a_R 
\end{split}
\end{equation}
where $\mu>0$ is a numerical constant to be chosen later and
\[
 \mathcal Q_{v,k} \coloneqq\Big\{R\in\mathcal Q_v:\, \mu k \leq \avgint_R \sum_{\mathcal Q\ni Q\leq R}a_Q\frac{\ind_Q}{|Q|}< \mu (k+1)\Big\}.
\]
Thus the proof reduces to proving a suitable estimate for $\mass_a (\mathcal Q_{v,k})$ for every fixed $v\in V$ and every $k> \mu(\log \#V)$. The next remark encapsulates  some simple but important geometric observations that are at the heart of the argument. 

\begin{remark}\label{r:geom}
Fix $R=P(I_R,K_R,v_R)\in \mathcal P_{V}$ and consider a $d$-plate $\mathcal P_V\ni Q=P(I_Q,K_Q,v_Q)\leq R$. Note that if $v_Q\neq v_R$ then $\pl_Q\cap \pl_R$ is a $(d-1)$-dimensional affine subspace and let $\mathsf{lin}_{Q,R}$ be the subspace parallel to $\pl_Q\cap \pl_R$. 
As $v_Q,v_R$ lie in a small neighborhood of $e_n ^\perp$, the subspace  $\pr_{e_n ^\perp} \mathsf{lin}_{Q,R}$ has dimension $d-1$ as well and is a  codimension 1 subspace of $e_n ^\perp$. We may thus pick an orthonormal basis $G_{Q,R}\coloneqq (g_1,\ldots,g_{d})$ for $e_n^\perp$ such that  $(g_1,\ldots,g_{d-1})$ is  an orthonormal basis of $\pr_{e_n ^\perp}\mathsf{lin}_{Q,R}$.

Let $\widehat{I}_{Q,R}\subset   $ be the smallest $d$-dimensional  cube in the coordinates $G_{Q,R}$ that contains $I_Q$. Then
\[
\widehat{I}_{Q,R} \subset e_n ^\perp, \qquad I_Q\subseteq  \widehat{I}_{Q,R} \subseteq d I_Q,
\]
where the dilation is taken with respect to the center of $I_Q$. Furthermore, defining 
\[
\widehat Q \coloneqq P(\widehat{I}_{Q,R} ,K_Q,v_Q)
\]
then $Q\subseteq \widehat Q$ and $\big|\widehat{Q}\big|\eqsim_n |Q|$. In case $v_q=v_R$ we just set $\widehat Q \coloneqq Q$ for the sake of having a general definition. Finally, setting  $\widehat I_{R}\coloneqq dI_R$ and $\widehat R \coloneqq P( \widehat I_{R},K_R,v_R)$ yields
\[ 
v_{\widehat R}=v_R,\qquad \widehat R  \supseteq R, \qquad  I_{\widehat R} \supseteq I_{\widehat Q}   \quad \forall Q\leq R.
\]
As $v_Q=v_{\widehat Q}$ and  $\pl_{\widehat Q}=\pl_Q$,   the plate $\widehat Q$ is a rotation and  $\pl_Q$-tangential dilation of $Q$ with respect to the line $\{c_Q+tv_Q:t\in v_Q\}$. Also, our construction yields that one of the $(d-1)$-dimensional edges of $\widehat Q$ lies on an affine copy of $\pl_Q\cap \pl_R$.	Note also that $\widehat I_{Q,R}$ depends both on $Q$ and $R$; we will however many times suppress the $R$-dependence as $R$ will be fixed and just write $\widehat I_Q$ in place of $\widehat I_{Q,R}$.
\end{remark}
With these definitions and conventions in hand,  we state a  geometric slicing lemma that will be important for the proof of Theorem~\ref{thm:carleson}.
\begin{lemma}\label{lem:slice} Let $R=P(I_R,K_R,v_R)\in \mathcal D_v$ and $\mathcal D_V\ni Q=P(I_Q,K_Q,v_Q)\leq R$, and $\widehat R, \widehat Q$ be  as in Remark \ref{r:geom}. Let $K\subset \R$ be an interval with $\pr_{e_n}(\widehat Q)\nsubseteq 3K$ and  $K_R\subseteq K.$ Then
\[	
 \max_{a\in \R }\frac{\big|\widehat Q \cap p(I_{\widehat R},a,v_{\widehat R})  \big|}{\big| I_{\widehat R}\hspace{.1em}\big|}
\lesssim_n \frac{1}{  \big| I_ {\widehat R} \hspace{.1em}\big|  | K |}\int_{ I_ {\widehat R} \times 3K} \ind_{\widehat Q}
	\]
with implicit constant depending only upon dimension.
\end{lemma}

\begin{proof} 
By composing $d$ shearing transformations, we reduce to the case of $v_{\widehat R}=e_n$. In this case the slices may simply be described by $p(I_{\widehat R},a,v_{\widehat R}) =I_ {\widehat R} \times \{a\}$. The conclusion is immediate if $v_Q=v_{\widehat R}=e_n$ since then the slices of $Q$ by planes perpendicular to $e_n$ have all constant measure. Thus we assume that $e_n=v_{\widehat R}\neq v_Q$. Then, $\pl_Q\cap \pl_R$ is an affine space of dimension $d-1$ parallel to the subspace $\mathsf{lin}_{Q,R}$. As $v_{\widehat R}=e_n$, we have  $\mathsf{lin}_{Q,R}= \pr_{e_n^\perp} \mathsf{lin}_{Q,R}$, in other words $\mathsf{lin}_{Q,R}\subset e_n^\perp.$ Setting $H(t)\coloneqq t+ (\mathsf{lin}_{Q,R})^\perp$ for $t\in\pl_R\cap \pl_Q\cap \widehat Q$, we have for any $a\in\R$ that
	\[
	\big|\widehat {Q}\cap ( I_ {\widehat R} \times \{a\})\big| = \int\displaylimits_{ \pl_R\cap \pl_{Q}\cap \widehat Q} \big|H(t)\cap \widehat {Q}\cap (I_{\widehat R}\times\{a\})\big|\, \d t.
	\]
Now observe that for each  $  t\in\pl_R\cap \pl_Q\cap \widehat Q$ the set $H(t) \cap \widehat{Q}$ is a  two-dimensional parallelogram lying on $H(t)$ with long side  perpendicular to $v_Q$ and short side of length $|K_Q|$ parallel to $e_n$. Our assumptions yield
\[
\pr_{e_n}(\widehat Q)\cap K\neq \varnothing,\qquad \pr_{e_n}(\widehat Q)\nsubseteq 3K;
\]
a two-dimensional calculation then reveals that there exists a set $A\subset 3K\setminus K$ with $|A|\geq |K|/3$ such that for each $a'\in A$ and all  $  t\in\pl_R\cap \pl_Q\cap \widehat Q$
\[
|H(t)\cap \widehat Q \cap (I_{\widehat R}\times \{a'\})|\eqsim \max_{a\in\R} |H(t)\cap \widehat Q \cap (I_{\widehat R}\times \{a\})|;
\]
the implicit constant in the estimate above is independent everything and in particular this estimate holds uniformly in $t$. This clearly implies that 
\[
\max_{a\in\R} |H(t)\cap \widehat Q \cap (I_{\widehat R}\times \{a\})| \leq \avgint_{A}   |H(t)\cap \widehat Q \cap (I_{\widehat R}\times \{a'\})|\, \d a'  \lesssim \avgint_{3K} \big|\widehat Q \cap H(t)\cap   (I_{\widehat{R}} \times \{a'\})\big|\,\d a'.
\]
The conclusion of the lemma readily follows by noticing that for every $a\in\R$ the quantity $|H(t)\cap \widehat Q \cap (I_{\widehat R}\times \{a\})|$ is independent of $t$ and integrating for $t\in \pl_R\cap \pl_Q\cap \widehat Q$.
\end{proof}

\vspace{1em}
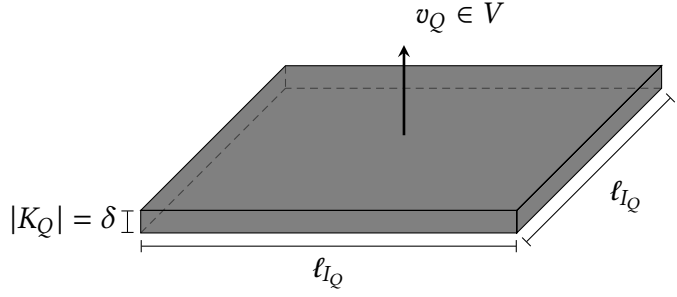
\begin{figure}
\centering
\begin{tikzpicture}[every edge quotes/.append style={auto, text=black}]
  \pgfmathsetmacro{\cubex}{5}
  \pgfmathsetmacro{\cubey}{.3}
  \pgfmathsetmacro{\cubez}{5}
  \draw [draw=black, every edge/.append style={draw=black, densely dashed, opacity=.5}, fill=gray]
    (0,0,0) coordinate (o) -- ++(-\cubex,0,0) coordinate (a) -- ++(0,-\cubey,0) coordinate (b) edge coordinate [pos=1] (g) ++(0,0,-\cubez)  -- ++(\cubex,0,0) coordinate (c) -- cycle
    (o) -- ++(0,0,-\cubez) coordinate (d) -- ++(0,-\cubey,0) coordinate (e) edge (g) -- (c) -- cycle
    (o) -- (a) -- ++(0,0,-\cubez) coordinate (f) edge (g) -- (d) -- cycle;

\path [every edge/.append style={draw=black, |-|}]
    (b) +(0,-5pt) coordinate (b1) edge ["$\ell_{I_Q}$"'] (b1 -| c)
    (b) +(-5pt,0) coordinate (b2) edge ["$|K_Q|=\delta $"] (b2 |- a)
    (c) +(3.5pt,-3.5pt) coordinate (c2) edge ["$\ell_{I_Q}$"'] ([xshift=3.5pt,yshift=-3.5pt]e)
    ;
\draw[line width=1pt,black,-stealth](-1.5,1)--(-1.5,2.2) node[anchor=south west]{$v_Q\in V$};
\end{tikzpicture}
\caption{A plate $Q=P(I_Q,K_Q,v_Q)$ in $\R^3$ perpendicular to $v_Q\in V$.}
\end{figure}

We now return to the estimate for $\mass_a (\mathcal Q_{v,k})$. Letting $R\in\mathcal Q_{v,k}$ we note that for every $\mathcal Q \ni Q\leq R$ we have $\widehat Q \supseteq Q$ and $|\widehat Q|\eqsim_n |Q|$ so that
\[
\mu k \leq \avgint_R \sum_{\mathcal Q\ni Q\leq R}a_Q\frac{\ind_Q}{|Q|}\lesssim_n \avgint_R \sum_{\mathcal Q\ni Q\leq R} a_Q \frac{\ind_{\widehat{Q}}}{|\widehat Q|}\lesssim_n \avgint_{\widehat R} \sum_{\mathcal Q\ni Q\leq R} a_Q \frac{\ind_{\widehat{Q}}}{|\widehat Q|}
\]
since $R\subseteq \widehat R$ and $|R|\simeq_n \big|\widehat R\big|$. For a collection of plates $\mathcal L\subseteq \mathcal D _V$ we define
\[
B_R ^{\mathcal L} \coloneqq \avgint_{\widehat R} \sum_{\substack{Q\in\mathcal L\\Q\leq R}} a_Q \frac{\ind_{\widehat Q}}{|\widehat Q|}.
\]
Then the considerations above imply that
\begin{equation}\label{eq:hatQ}
\mathcal Q_{v,k} \subseteq \widehat {\mathcal Q}_{v,k}\coloneqq \{Q\in\mathcal Q_v:\, B_R ^{\mathcal Q} >c_n \mu k\}
\end{equation}
for some dimensional constant $c_n>0$ and $\mu$ to be chosen.

The proof of Theorem~\ref{thm:carleson} relies upon the estimate for $\sh(\widehat{\mathcal Q}_{v,k})$ contained in the following lemma.
\begin{lemma}\label{lem:expdecay} Let $\delta>0$, $V\subset\mathbb S^d$ be a finite set of directions, and $\mathcal Q \subseteq \mathcal D_{V}$ satisfy the assumptions of Theorem~\ref{thm:carleson}. For $v\in V$ and $k\geq 1$ we define $\widehat{\mathcal Q}_{v,k}$ as in \eqref{eq:hatQ} above, with $\mu$ a sufficiently large dimensional constant. Then
\[
\big|\sh(\widehat{\mathcal Q}_{v,k})\big|\lesssim_n 2^{-k} \mass_a (\mathcal Q).
\]	
\end{lemma}
With Lemma~\ref{lem:expdecay} in our disposal we can complete the proof of the main result of this paragraph.

\begin{proof}[Proof of Theorem~\ref{thm:carleson}] Let $\mathcal Q\subseteq \mathcal D_{V}$ be a collection of plates such that the operators $\{\M_{\mathcal Q_v}:\, v\in V\}$ are of weak-type $(1,1)$, uniformly in $v\in V$, with weak $(1,1)$ bound depending only on the dimension. By \eqref{eq:mainsplit} combined with \eqref{eq:hatQ} and the estimate of Lemma~\ref{lem:expdecay} we have that
\[
\begin{split}
\|T_{\mathcal Q}(a)\|_{L^2(\R^n)} ^2 &\lesssim \mu \Big((\log \#V) +( \#V)\sum_{k\geq \mu (\log\#V)} k 2^{-k}\Big) \mass_a(\mathcal L) \lesssim \mu (\log\#V)\mass_a(\mathcal L)
\end{split}
\]
so that 
\begin{equation}\label{eq:recurse}
\|T_{\mathcal Q}(a)\|_{L^2(\R^n)}  \leq C_n \mu^{\frac 12} (\log\#V)^{\frac12} \mass_a(\mathcal Q) ^{\frac12}
\end{equation}
with $\mu$ as in the assumption of Lemma~\ref{lem:expdecay}. Note that this means that $\mu$ can be chosen to be a dimensional constant and this completes the proof of the theorem.
\end{proof}
It remains to prove Lemma~\ref{lem:expdecay} which follows by an iterative application of the lemma below, as in the proof of \cite{APHPR}*{Lemma 2.21}.

\begin{lemma}{\label{lem:recurse}} Let $V\subset \mathbb S^d$ be a finite set of directions and $\mathcal L \subseteq \mathcal R \subseteq \mathcal D_{V}$ such that for every $L\in\mathcal L$ there exists $R\in\mathcal R$ with $L\leq R$. Furthermore we assume that
	\[
	\sup_{v\in V} \|\M_{\widehat{\mathcal R}_v}:\, L^1(\R^n)\to L^{1,\infty}(\R^n)\|\lesssim_n 1.
	\]
 Fix some $v\in V$ and let $\mu$ be a sufficiently large dimensional constant. There exists $\mathcal L_1 \subset \mathcal L $ such that 
\begin{itemize}
	\item [(i)] $\displaystyle\mass_a (\mathcal L_1) \leq \frac12 \mass_a(\mathcal L)$;
	\item [(ii)] For all plates $R\in\mathcal R_v$ such that $B_R ^{\mathcal L}>\mu$ we have
	\[
	B_R ^{\mathcal L} \leq \mu + B_R ^{\mathcal L_1}.
	\]
\end{itemize}
\end{lemma}

\begin{proof}  Without loss of generality we can assume that $v=e_n$. Let $\mathcal R' _v$ denote the collection of $R\in\mathcal R_v$ with $B_R ^{\mathcal L}>\mu$. For the plates $R\in\mathcal R_v '$ and plates $Q\in\mathcal L$ with $Q\leq R$ we define $\widehat{R},\widehat{Q}$ as in Remark~\ref{r:geom}. We remember also the notation $\widehat R=P(\widehat{I}_R,K_R,v_R)$ and $Q=P(\widehat {I}_Q,K_Q,v_Q)$; since we will always consider the case $Q\leq R$ we have that $\widehat{I}_Q\subseteq \widehat {I}_R$.
	
Given some interval $K\subset \R$ and $R\in\mathcal R' _v$ we define the collections
	\[
	\mathcal B_{R,K} ^{\mathrm{in}}\coloneqq \{Q\in\mathcal L:\, Q\leq R,\quad \pi_{e_n}(\widehat{Q})\subseteq  3K \},\qquad 	\mathcal B ^{\mathrm{out}} _{R,K}\coloneqq \{Q\in\mathcal L:\, Q\leq R,\quad \pi_{e_n}(\widehat{Q})\nsubseteq 3 K \}.
	\]
We define
\[
B^{\mathrm{in}} _{R, K} \coloneqq \sum_{Q\in \mathcal B_{R,K} ^{\mathrm{in}}} a_Q\frac{|Q'\cap (\widehat{I}_{R}\times K)|}{\big|\widehat{Q}\big|\big|\widehat{I}_{R}\times K\big|},\qquad B_{R, K} ^{\mathrm{out}}\coloneqq  \sum_{Q\in \mathcal B_{R,K} ^{\mathrm{out}}} a_Q \frac{\big|\widehat{Q}\cap (\widehat{I}_{R}\times K)\big|}{\big|\widehat{Q}\big|\big|\widehat{I}_{R}\times K\big|},
\]
and note that for any $K$ we have the splitting
\[
B_R ^{\mathcal L} = \avgint_{\widehat R} \sum_{Q\in\mathcal B_{R,K} ^{\mathrm{in}}}a_Q \frac{\ind_{\widehat{Q}}}{\big|\widehat{Q}\big|}+\avgint_{\widehat{R}}\sum_{Q\in\mathcal B_{R,K} ^{\mathrm{out}}}a_Q\frac{\ind_{\widehat{Q}}}{\big|\widehat{Q}\big|}.
\]

\subsubsection*{Easy case:} Let $\mathcal R_1$ be the collection of those $R\in\mathcal R' _v$ such that $B_{R,K_R}  ^{\mathrm {out}}\leq \mu$. Then we have for $R\in\mathcal R_1$ that
\[
B_R ^{\mathcal L} \leq \mu + B_{R,K_R} ^{\mathrm{in}} =\mu +B_R ^{\mathcal L' _1},\qquad \mathcal L_1 ' \coloneqq \bigcup_{\rho\in\mathcal R_1}\mathcal{B}_{\rho,K_\rho} ^{\mathrm{in}}.
\]
Noting that for all $Q\in \mathcal L' _1 $ we have that $Q\subseteq \widehat{Q}  \subseteq \widehat I_{\rho}\times 3K_\rho$ for some $\rho\in\mathcal R_1$, we get
\[
\begin{split}
\mathrm{mass}_{a,1}(\mathcal L' _1)\coloneqq\sum_{Q\in\mathcal L' _1 } a_Q \leq  \bigg|\bigcup_{\rho\in\mathcal R_1}\bigcup_{Q\in \mathcal B_{\rho, K_\rho} ^{\mathrm{in}} }Q\bigg|\leq  \bigg|\bigcup_{\rho\in\mathcal R_1} (\widehat I_{\rho}\times 3K_\rho)\bigg|.
\end{split}
\]
Furthermore for every $\rho\in\mathcal R_1$ we have
\[
\mu<B_\rho ^{\mathcal L}= \avgint_{\widehat \rho }\sum_{\substack{Q\leq \rho\\Q\in\mathcal L}} a_Q \frac{\ind_{\widehat Q}}{\big|\widehat Q\big|}.
\]
As all the $\rho\in \mathcal R_1$ have fixed orientation our assumption entails that the maximal operator $\M_{\widehat{\mathcal R}_{v}}=\M_{\widehat{\mathcal R}_{e_n}}$ is of weak type $(1,1)$ with constant depending only on the dimension. Since
\[
\bigcup_{\rho\in\mathcal R_1}( \widehat I_\rho\times 3K_\rho) \subseteq \Big\{\M_{\widehat{\mathcal R}_{v}} \Big(\sum_{Q\in\mathcal L}a_Q\frac{\ind_{\widehat{Q}}}{\big|\widehat Q\big|}\Big)\gtrsim_n \mu \Big\},
\]
we conclude that $\mass_{a}(\mathcal L' _1)\leq \mass_{a}(\mathcal L)/4$ if $\mu>0$ is chosen to be a sufficiently large dimensional constant.

\subsubsection*{The main case:} Here we consider $R\in\mathcal R_2\coloneqq \mathcal R_v '\setminus \mathcal R_1$. Let us write again $\widehat R=\widehat{I}_R\times K_R$ and consider the intervals $J$ of the form $J=3^\ell K_R$ for $\ell\geq 0$ such that $B_{R,J} ^{\mathrm{out}}>\mu$. Since $B_{R, K_R} ^{\mathrm{out}}>\mu$ for $R\in\mathcal R_2$ the maximal such $J$ which we call $J_R$ will contain $K_R$ and $B_{R, 3J_R} ^{\mathrm{out}}\leq\mu$. 

By Lemma~\ref{lem:slice} we have for each $a\in 3J_R$ and $Q\in \mathcal B_{R,J_R} ^{\mathrm{out}}$
\[
\big|\widehat{Q}\cap (\widehat I_{R}\times \{a\})\big|\lesssim \frac{1}{|3J_R|}\int _{\widehat I_{R}\times 3J_R} \ind_{\widehat Q}.
\]
We can then calculate
\[
\begin{split}
\sum_{Q\in\mathcal B_{R,J_R} ^{\mathrm{out}}} a_Q \frac{\big|\widehat{Q}\cap  ( \widehat{I}_{R} \times \{a\})\big|}{\big|\widehat{Q}\big||\widehat {I}_{R}\times \{a\}|}& \lesssim \sum_{Q\in\mathcal B_{R,J_R} ^{\mathrm{out}}} a_Q  \frac{\big|\widehat{Q} \cap ( \widehat{I}_{R} \times 3J_R)\big|}{\big|\widehat{Q}\big|\big|\widehat I_{R}\times 3J_R\big|}
\\
&\lesssim B^{\mathrm{out}} _{R, 3J_R }+\avgint_{(\widehat{ I}_R\times 3 J_R)}\sum_{Q\in\mathcal B_{R,J_R} ^\mathrm{out} \setminus \mathcal B^{\mathrm{out}} _{R,3J_R} } a_Q \frac{\ind_{\widehat Q}}{|\widehat Q|}
\\
& \leq \mu +  \sum_{Q\in\mathcal B_{R,J_R} ^\mathrm{out} \setminus \mathcal B^{\mathrm{out}} _{R,3J_R} } a_Q\frac{\big|\widehat{Q} \cap (\widehat{I}_R\times 3J_R)\big|}{\big|\widehat{Q}\big|\big|\widehat{I}_R\times 3J_R\big|};
\end{split}
\]
in passing to the last line we used the maximality of $J_R$. Now all the plates $\widehat Q$ appearing in the sum of the right hand side in the estimate above are contained in $\widehat{I}_R \times 9J_R$ and so the sum of the second summand above is estimated by a dimensional constant $c_n>1$ so that
\[
\sum_{Q\in\mathcal B_{R,J_R} ^{\mathrm{out}}} a_Q \frac{\big|\widehat{Q}\cap  ( \widehat{I}_R \times \{a\})\big|}{\big|\widehat{Q}\big|\big|(\widehat{I}_R\times \{a\}\big|}\lesssim_n\mu+c_n \lesssim_n \mu
\]
if $\mu$ is sufficiently large depending only upon dimension. Since $J_R\supseteq K_R$ we can integrate for $a\in K_R=\widehat K_{R}$ to conclude that
\[
\sum_{Q\in\mathcal B_{R,J_R} ^{\mathrm{out}}} a_Q \frac{\big|\widehat Q\cap \widehat{R}\big|} {\big|\widehat Q\big|\big|\widehat R\big|}\leq \kappa_n\mu
\]
for some dimensional constant $\kappa_n>1$. This shows that for $\mu$ sufficiently large depending upon dimension we have
\[
B^{\mathcal L} _R\leq\kappa_n\mu+ \sum_{Q\in B_{R,J_R} ^{\mathrm{in}}} a_Q \frac{\big|\widehat{Q}\cap \widehat{R}\big|}{\big|\widehat{Q}\big|\big|\widehat{R}\big|} = 
  \kappa_n \mu+B_R ^{\mathcal L' _2},\qquad  \mathcal L' _2\coloneqq \bigcup_{\rho \in \mathcal R_2} \mathcal B_{\rho,J_\rho} ^{\mathrm{in}}.
\]
with
\begin{equation}\label{eq:massL2}
\sum_{Q\in\mathcal L' _2}a_Q \leq \bigg| \bigcup_{\rho \in\mathcal R_2} \widehat{I}_\rho\times 3J_\rho\bigg|.
\end{equation}
By the previous estimates we have that for each $\rho\in \mathcal R_2$
\[
\begin{split}
\mu<B^{\mathrm{out}} _{ \rho,J_\rho}&=\sum_{Q\in \mathcal B_{\rho,J_\rho} ^{\mathrm{out}}}a_Q\frac{\big|\widehat{Q} \cap (\widehat{I}_{\rho}\times J_\rho)\big|}{\big|\widehat{Q}\big|\big|\widehat{I}_{\rho}\times J_\rho\big|}= 
 \avgint_{J_\rho}  \sum_{{Q\in \mathcal B_{\rho,J_\rho} ^{\mathrm{out}}}}a_Q\frac{\big|\widehat{Q} \cap (\widehat{I}_{\rho}\times \{a\})\big|}{\big|\widehat{Q}\big|\big|\widehat{I}_{\rho}\times \{a\}\big|}\, \d a
\\
&\leq\frac{ \kappa_n \mu}{|J_\rho|} |\{a\in J_\rho:\,\psi_{\rho}(a)>\mu/2\}|+\mu/2
\end{split}
\]
with 
\[
\psi_\rho(a)\coloneqq \sum_{Q\in \mathcal B_{\rho,J_\rho} ^{\mathrm{out}}}a_Q\frac{\big|\widehat{Q}  \cap (\widehat I_{\rho}\times \{a\})\big|}{\big|\widehat Q||\widehat I_{\rho}\times\{a\}\big|}.
\]
Thus there exists a set $J' _\rho \subseteq J_\rho$ with $|J' _\rho|\gtrsim_n |J_\rho|$ so that $\psi_\rho(a)>\mu/2$ for $a\in J' _\rho$. Now note that for $a\in J' _\rho$
\[
\frac{\mu}{2}\leq \psi_\rho(\alpha)\leq  \avgint_{\widehat{I}_\rho\times\{a\}}\sum_{Q\in \mathcal B_{\rho,K_\rho} ^{\mathrm{out}}} a_Q\frac{\ind_{\widehat Q}}{\big|\widehat Q\big|} \leq \inf_{x\in \widehat I_\rho} \M_{v^\perp}\Big(\sum_{\substack{Q\in \mathcal L\\Q\leq \rho}}a_Q\frac{\ind_{\widehat Q}}{\big|\widehat Q\big|}\Big)(x,a).
\]
In the estimate above we write $\M_{v^\perp}$ for the maximal  function
\[
\M_{v ^\perp} f(x)\coloneqq \sup_{s>0} \avgint_{Q_{v^\perp} (0,s)} |f(x+t)|\,\d t,\qquad x\in\R^n,
\]
where $Q_{v^\perp}(0,s)$ denotes the cube in $v^\perp\simeq \R^d$ with sidelength $s>0$ and centered at $0\in v^\perp$. Note that for $v\in\mathbb S^d$ the operator $\M_{v^\perp}$ is of weak-type $(1,1)$, uniformly in $v$. Thus
\[
\bigcup_{\rho\in\mathcal R_2} (\widehat  I_\rho \times J' _\rho) \subseteq S\coloneqq \Big\{z\in\R^n:\,  \M_{v^\perp}\Big(\sum_{Q\in \mathcal R}a_Q \frac{\ind_{c_2Q}}{|c_2Q|}\Big)(z)>\mu/2 \Big\}
\]
On the other hand since $J_\rho ' \subseteq J_R$ with $|J_{\rho} '|\gtrsim |J_\rho|$ we readily see that 
\[
\bigcup_{\rho\in\mathcal R_2} (\widehat I_\rho \times J_\rho) \subseteq \Big\{z\in \R^n:\, \M_{e_n}\big(\ind_{\cup_{\rho\in\mathcal R_2} \widehat I_\rho \times J' _\rho}\big)(z)\gtrsim_n 1   \Big\}\subseteq \big\{z\in\R^n:\, \M_{e_n}(\ind_S )(z)\gtrsim_n 1\big\}.
\]
Combining \eqref{eq:massL2} with the weak type $(1,1)$ inequalities for $\M_{e_n}$ and $\M_{v^\perp}=\M_{e_n ^\perp}$ (since we assume that $v=e_n$) and choosing $\mu$ to be a sufficiently large dimensional constant 
\[
\mathrm{mass}_{a,1}(\mathcal L' _2)=\sum_{Q\in\mathcal L' _2}a_Q \leq \frac{C}{\lambda} \sum_{Q\in\mathcal L }\leq  \frac14 \sum_{Q\in\mathcal L}a_Q = \mathrm{mass}_{a,1}(\mathcal L)/4.
\]
Now we set $\mathcal L '\coloneqq \mathcal L' _1 \cup \mathcal L' _2$ and the proof is complete.
\end{proof}

\subsection{Application to a maximal function estimate} As an immediate application of the directional Carleson embedding theorem for plates we describe below a sharp theorem for maximal averages along codimension $1$ plates. Let $d=n-1$ and consider $\sigma\in\mathrm{Gr}(d,n)=\mathrm{Gr}(d,d+1)$. We remember that the codimension $1$ averages at scale $s>0$ of a function $f\in\mathcal S(\R^{d+1})$ can be given in the form
\[
\langle f \rangle_{s,\sigma} (x) \coloneqq \int\displaylimits_{B_{d+1}(s)\cap \sigma} f(x-y) \frac{\d y}{s^{d}}, \qquad x\in \R^{d+1},
\]
where $B_{d+1}(s)$ denotes the ball of radius $s$ and centered at $0\in\R^{d+1}$. Given a finite subset $\Sigma\subset \mathrm{Gr}(d,d+1)$ we are interested in the corresponding maximal averaging operator along codimension $1$ plates given by $\Sigma$
\[
\M_{\Sigma} f(x)\coloneqq \sup_{s>0}\sup_{\sigma\in\Sigma} \l |f|\r_{\sigma,s}(x),\qquad x\in\R^{d+1}.
\]
As a consequence of the directional Carleson embedding theorem we obtain the sharp bounds for $\M_\Sigma$ for arbitrary finite $\Sigma\subset\mathrm{Gr}(d,d+1)$.

\begin{proof}[Proof of Theorem~\ref{thm:codim1_intro}] We write $n=d+1$ throughout the proof. It suffices to prove the weak-type $(2,2)$ estimate. Indeed the $L^p$-estimate will then follow by interpolation between the $L^2(\R^n)\to L^{2,\infty}(\R^n)$ and $L^\infty(\R^n)\to L^\infty(\R^n)$ bounds. Furthermore the strong-type $(2,2)$ estimate follows by the corresponding weak-type estimate with an additional $\sqrt{\log\#\Sigma}$-loss by the well known interpolation argument of Str\"omberg, \cite{Stromberg} as in the proof of Theorem~\ref{thm:L2nik}.

For the $L^2(\R^n)\to L^{2,\infty}(\R^n)$ note that is suffices to prove the weak-type $(2,2)$ estimate with the same dependence on $\#\Sigma$ for the closely related dyadic maximal operator
\[
\M_{V,\delta}f(x)\coloneqq \sup_{\substack{ Q\in\mathcal D_{V,\delta} \\ Q\ni x}}\, \avgint_Q |f|,\qquad x\in\R^n,
\]
for $V\subset \mathbb S^d$ finite and fixed $\delta>0$, and with a bound independent of $\delta$. This operator can be linearized as in the proof of Theorem~\ref{thm:L2nik} in the form
\[
T_{\mathcal Q}f(x)\coloneqq \sum_{Q\in\mathcal Q} \Big(\,\avgint_Q |f|\Big)\ind_{F_Q}(x),\qquad x\in \R^n
\]
where $\mathcal Q\subset \mathcal D_{V,\delta}$ a finite collection of $(d,\delta)$ plates and $\{F_Q\}_{Q\in\mathcal Q}$ a pairwise disjoint collection of measurable sets with $F_Q\subseteq Q$ for every $Q\in\mathcal Q$. Denoting by $T_{\mathcal Q} ^*$ the adjoint of $T_{\mathcal Q}$ we have that 
\[
\|\M_{V,\delta}\|_{L^2(\R^n)\to L^{2,\infty}(\R^n)} =\sup_{\mathcal Q\subset \mathcal D_{V,\delta}}\sup_{0<|E|<\infty}\frac{\|T^* _{\mathcal Q}(\ind_E)\|_{L^2(\R^n)}}{|E|^{\frac12}}
\]
where 
\[
T_{\mathcal Q} ^* (\ind_E)(x) = \sum_{Q\in\mathcal Q} |F_Q\cap E| \frac{\ind_Q}{|Q|}.
\]
Clearly $a=\{a_Q\}_{Q\in\mathcal Q}=\{|F_Q\cap E|\}_{Q\in\mathcal Q}$ is a Carleson sequence in the sense of Definition~\ref{def:carleson} so the required estimate for $T^* _{\mathcal Q}(\ind_E)$ follows by a straightforward application of Theorem~\ref{thm:carleson}.

The fact that these estimates are best possible follows by considering $A$ to be a Kakeya collection of $\delta \times 1$-tubes in $\R^2$ and taking $
A'\coloneqq A \times [-1,1]^{n-2}$. Now for each tube in $A$ we can consider a  $1 \times \delta ^{n-2}$ plate that contains the tube and is perpendicular to the copy of $\R^2$ that contains $A$. Calculating the averages of $\ind_{A'}$ with respect to these plates yields the  sharpness of the weak-type $(2,2)$ estimate and the sharpness of the strong $(p,p)$ estimate for $p>2$. Note that the numerology here is $\#\Sigma=\#V\simeq 1/\delta$. The optimality of the strong $(2,2)$-estimate follows similarly by considering a function in $\R^2$ that yields the sharpness of the $2$-dimensional results and extending them in $\R^n$ by taking a tensor product with a smooth bump in $\R^{n-2}$; see also Remark~\ref{rmrk:2dsharp}.
\end{proof}

\begin{remark} We stress an important switch in our point of view when proving estimates for the maximal operator $\mathrm{M}_{V,\delta}$ above, compared to say the corresponding estimates for the Nikodym operator $\mathcal N_\delta$ in \S\ref{s:kakeya}. Indeed although these two operators appear to be quite similar, in the case of $\mathrm{M}_{V,\delta}$ we are interested in proving estimates for \emph{arbitrary finite subsets} $V\subset \mathbb S^d\eqsim \mathrm Gr(d,d+1)$. Thus our $\delta$-fattening of the thin plates $B_n(s)\cap \sigma$ for $\sigma\in \mathrm Gr(d,d+1)$ is purely qualitative, it is there just to allow us to use $\delta$-plates which have positive measure in $\R^{d+1}$ and are more amenable to geometric arguments. These estimates are to be $\delta$-free as we use a limiting argument in order to recover thin plates as $\delta\to 0^+$. Necessarily, for this argument the cardinality $\#V$ and $\delta$ are completely independent of each other. This is in contrast to the geometric setup underlying the definition of the Nikodym operator $\mathcal N_\delta$ where the implicit subset of $\mathrm Gr(d,d+1)$ is a $\delta$-net and thus has cardinality $\sim \delta^{-d}$, namely the thickness of the plates and the  set of essentially directions present in $\mathcal N_\delta$ are intimately connected.
\end{remark}

\subsection{Application to a conical frequency square function}\label{sec:sf} We describe below a  square function estimate in the spirit of Rubio de Francia, given with respect to conical frequency projections in $\R^{d+1}$. Let $\{B_k\}_{k=1} ^N$ be a collection of open balls   in $\mathbb R^{d+1}$ whose centers $v_k$ lie on $\mathbb S^d$, and is \emph{well-distributed} in the sense that
\[
\sum_{k=1}^N \ind_{3B_k} \lesssim 1
\]
Here, as usual,  $3B_k$ is a threefold dilation of the Euclidean ball $B_k$ with respect to its center. For each $k$, let 
\[
\phi_k\in\mathcal S(\mathbb R^{d+1}), \qquad \cic{1}_{B_k} \leq    \phi_{k} \leq  \cic{1}_{3B_k}\]
and define the conical frequency projection 
\[
S_{k}f(x)\coloneqq  \int_{\mathbb S^d } \int_0 ^\infty \hat f (r \xi') \phi_{k}(\xi') e^{ix\cdot r \xi'} r^d \,\d r\, \d\sigma_{d}(\xi'),\qquad x\in  \R^{d+1}.
\]
The  $d$-plate Carleson embedding Theorem \ref{thm:carleson} may be used to deduce a  square function estimate  for the projections $S_k$  with sharp dependence on the parameter $N$. 

\begin{theorem} \label{t:sfe} For $2\leq p<4$,
	\[
	\Big\|\Big(\sum_k |S_k f|^2\Big)^{\frac12}\Big\|_{L^p(\R^{d+1})}\lesssim_{p,d} (\log N)^{\frac{1}{2}-\frac{1}{p}} \|f\|_{L^p(\R^{d+1})}.
	\]
Furthermore, the  restricted $L^4(\R^n)$-type estimate  \[
	\Big\|\Big(\sum_k |S_{k}(f\ind_E)|^2\Big)^{\frac12}\Big\|_{L^4(\R^{d+1})}\lesssim_{d}  (\log N)^{\frac 14} \|f\|_{L^\infty(\R^{d+1})}|E|^{\frac14}
	\]
	holds for all bounded measurable sets $E\subset \R^{d+1}$. These bounds are best possible up to the implicit numerical constants.
\end{theorem} 
For the optimality in the estimate of Theorem~\ref{t:sfe} we send to \cite[Section 8]{APHPR}, noting that the two-dimension bound becomes a lower bound in $\R^{d+1}$ for general $d$ by taking balls $\{B_k\}_{k=1} ^N$ have centers lying on a copy of $\mathbb S^1$ and functions $f$ which are suitable tensor products.

The remainder of this section contains the proof of the upper bounds in Theorem  \ref{t:sfe}. Below,   $\ell^2_N$ stands for the    Euclidean norm on $\mathbb C^N$. The first step consists of the radial decoupling 
\begin{equation}
\label{e:LPdec} 	\left\| S_{k} f\right\|_{L^p(\R^{d+1}; \ell^2_k)}\lesssim_{p,d}  \left\| S_{k} Q_{m} f\right\|_{L^p\left(\R^{d+1}; \ell^2_{N}\otimes \ell^2_{  \mathbb Z} \right)}
  \end{equation}
 where $p\in [2,\infty)$ and $\{Q_m:m\in \mathbb Z\}$ are  Fourier multiplier operators whose associated multipliers are radial and a partition of unity of $\R^{d+1}\setminus \{0\}$ subordinated to the finitely overlapping cover $A_{m}=\{\xi \in \R^{d+1}:2^{-m-1}< |\xi|<2^{-m+1}\}, m\in \mathbb Z$. The proof is a simple application of the weighted norm inequality
 \[
 \left\|g \right\|_{L^2(w)} \lesssim   \left\|   Q_{m} g\right\|_{L^2\left(\widetilde{\mathrm{M}}w;   \ell^2_{  \mathbb Z} \right)}
 \]
 where $\widetilde{\mathrm{M}}$ stands for the third iterate of the standard $(d+1)$-dimensional Hardy-Littlewood maximal operator
 see \cite{APHPR}*{Lemma 5.6} for details.

  Note that $S_{k} Q_{m}$ is supported in the frequency tube 
  \[
  \omega_{k,m}\coloneqq \left\{\xi\in A_m:\, \frac{\xi}{|\xi|}\in 3B_k \right\}
  \]
  whose center line is through the center $v_k $ of $B_k$ and whose spatial dual is the \emph{plate}
  \[
  R_{m,k}^0\coloneqq\{x\in \R^{d+1}:\,\, | \pr_{v_k} x| < 2^{m} \delta_k,\,  \pr_{v_k^\perp} x < 2^{m}  \}
  \]
  of \emph{eccentricity} $\delta_k$, the radius of $B_k$ and sidelength $2^{m}$. Let
   $\mathcal{R}_{m,k}$ be a $\lesssim_d 1$-overlapping cover of $\R^{d+1}$ by translates $R$ of $R_{m,k}^0$ and   $t\in \mathbf{T}_{m,k}$ be the collection of \emph{tiles} $t=(R_t, \omega_{t})$  with $\omega_{t}= \omega_{m,k}$ and $R\in\mathcal{R}_{m,k} $. The space-frequency projection on $t$ is represented by the intrinsic coefficient 
\[
a_t (f) = \sup_{\phi \in \Phi_t^M} |\langle f, \phi \rangle|^2
\]
where $\Phi_t^M\subset \mathcal{S}(\R^{d+1})$ is the class of functions whose frequency support is contained in $\omega_t$ and are uniformly spatially adapted to $R_t$ in the sense that
\[
\sqrt{|R_t|}|\phi| \leq \cic{1}_{R_t}  + \sum_{k=0}^\infty 2^{-Mk} \cic{1}_{2^{k+1}R_t\setminus 2^{k} R_t} 
\]   uniformly over $\phi \in \Phi_t^M$.
A standard space-frequency discretization, see e.g.\ \cite{APHPR}*{Sect.\ 5}  the right-hand side of \eqref{e:LPdec} is pointwise bounded by the discretized square function associated to the coefficients $a_{t}$, namely
\begin{equation}
\label{e:dsf}  \left\|S_{k} Q_m f \right\|_{\ell^2_k\otimes \ell^2_{\mathbb Z}} \lesssim \Delta f \coloneqq \left(\sum_{t\in \mathbf{T}} a_t(f) \frac{\cic{1}_{R_t}}{|R_t|}  \right)^{\frac12}
\end{equation}   
where $\mathbf{T}_k=\bigcup_{m\in \mathbb Z}\mathbf{T}_{m,k}$ and $\mathbf{T}=\bigcup_{1\leq k\leq N}\mathbf{T}_{k}.$ Theorem \ref{t:sfe} is thus reduced to the corresponding bounds for the discretized square function $\Delta f$. In fact, by standard restricted-type interpolation, it suffices to prove the  restricted type estimate  that follows.

\begin{proposition} \label{p:sfe} $\displaystyle
	\Big\|\Delta (f\ind_E)^2\Big\|_{L^2(\R^{d+1})}\lesssim_{d}  (\log N)^{\frac 12} \|f\|_{L^\infty(\R^{d+1})}|E|^{\frac12}.$
\end{proposition} 

\begin{proof} By a limiting argument, it suffices to replace the  universe of tiles $\mathbf T$ in the definition of $\Delta $ by a finite subcollection,  which we still call $\mathbf T$. We still denote by $  \mathbf T_{m,k}, \mathbf T_{k}$ the subcollections of $\mathbf T$ with sidelength and frequency parameters $m,k$.
 Recall that $n=d+1$. By linearity, we may restrict to the case $\|f\|_{\infty}=1$.  By finite splitting of the cones, we can assume all $v_k$ lie in a small neighborhood of $e_n$ as specified in Theorem \ref{thm:carleson}. Using the well known $3^d$-grid lemma and finite splitting of $\mathbf T$, we may find a dyadic grid $\mathcal D$ on $e_n^\perp$ such that for all $t\in \mathbf T_{m,k}$, 
  $\pr_{e_n^\perp} R_t\subset I_t$ for some $I_t\in \mathcal D$ with $\ell_I =2^{m+3} $. Let also $\mathcal K^j$, $j=0,1,2$ be a system of three shifted  dyadic grids on $\R$     and define
  \[
  \begin{split} &
  \mathcal Q_{m,k}\coloneqq \left\{P(I, K , v_k):\,\,   I \in \mathcal D, \ell_I =2^{m+3}, \, K \in \mathcal K^0 \cup \mathcal K^1 \cup \mathcal K^2 ,\, \ell_K\in [2^{m+3}\delta_k, 2^{m+4}\delta_k) \right\}, \\  &\mathcal Q_{k}\coloneqq \bigcup_{m\in \mathbb Z}   \mathcal Q_{m,k}, \qquad \mathcal Q \coloneqq \bigcup_{1\leq k \leq N}   \mathcal Q_{k}. 
  \end{split}
  \]
Notice that for each $t\in \mathbf{T}_{m,k}$ there exists at least $1$ and at most 3 elements $Q\in \mathcal Q_{m,k}$ with $R_t\subset Q$ and $|Q|\lesssim |R_t|$. Thus, setting
\[
\mathbf{T}(Q)\coloneqq \left\{t\in \mathbf{T}_{m,k}:\, R_t\subset Q\right\}, \qquad
a_{Q}\coloneqq \sum_{ t\in \mathbf{T}(Q)} a_t (f\cic{1}_E), \qquad Q\in \mathcal Q_{m,k}
\]
leads to the pointwise estimate $\Delta (f\cic{1}_E)^2 \lesssim   T_{\mathcal Q} (a)  $, with $T_{\mathcal Q}(a)$ in the form \eqref{eq:TQ}.

Proposition \ref{p:sfe} may then be obtained by an application of Theorem \ref{thm:carleson} to the union $\mathcal Q$  of the $N$ collections $\mathcal Q_k$. Notice that the plates of $\mathcal Q_k$ have fixed eccentricity $\delta_k$ and thus obey the weak (1,1) assumption of that theorem. We must then compute the directional Carleson norm of the sequence $\{a_Q: Q\in \mathcal Q\}$.  Firstly, from the finite frequency overlap of the Fourier supports of any collection $\{\phi_t\in \Phi_t^{M}: t\in \mathbf{T} \}$ and the spatial localization of the collection $\{\phi_t\in \Phi_t: t\in \mathbf{T}: \omega_{t}=\omega_{m,k} \}$ to a finitely overlapping collection $\mathcal R_{m,k}$, we gather that
\begin{equation}
\label{e:L2est}
\sum_{  t\in \mathbf{T} } |\langle g,\phi_t\rangle|^2\lesssim \|g\|_2^2
\end{equation}
whence
\[
\mathsf{mass}_a(\mathcal Q) \leq 2 \sum_{t\in \mathbf{T} }  a_t (f\cic{1}_E) \lesssim |E|.
\]
The details of this estimate are similar to those of \cite[Lemma 4.4]{APHPR}. 

We turn to the verification of the directional Carleson sequence property. That is for each  fixed $k$, let $\mathcal T\subset \mathcal P_{v_k}$, we need to prove that
\[
\sum_{L\in \mathcal L} a_L \lesssim |\mathsf{sh}(\mathcal T)|
\]
whenever $\mathcal L$ is a collection subordinated to $\mathcal T$.  If $\mathrm{M}$ stands locally for the maximal averaging operator over all plates $\mathcal P_{v_k}$, which is of type $(2,2)$ say,  define the enlargement of $\mathsf{sh}(\mathcal T)$
\[
U\coloneqq \left\{\mathrm{M} \cic{1}_{ \mathsf{sh}(\mathcal T)} > 2^{-10} \right\}
\]
so that $|U|\lesssim |\mathsf{sh}(\mathcal T)|$.  Then
\[
\sum_{L\in \mathcal L} a_L \leq \sum_{L\in \mathcal L}\sum_{t\in \mathbf{T}(Q) }   a_t (f\cic{1}_{E\cap U}) +  \sum_{L\in \mathcal L}\sum_{t\in \mathbf{T}(L) }   a_t (f\cic{1}_{E\cap U^c}). 
\]
The local part $f\cic{1}_{E\cap U}  $ is then dealt with  using \eqref{e:L2est} as follows:
\[
\sum_{t\in \mathbf{T}(L) }  a_t (f\cic{1}_{E\cap U}) \lesssim \|f\cic{1}_{E\cap U}\|_2^2 \lesssim |U| \lesssim |\mathsf{sh}(\mathcal T)|.
\]
To estimate the nonlocal part split $T\in \mathcal T$ into the union of collections $\mathsf T(u)$, saying $T\in  \mathsf{T}(u)$ if $u$ is the least integer such that $2^{u+1}T\cap U^c\neq \varnothing.$ A suitable version of Journ\'e's lemma \cite[Lemma 4.7]{APHPR} yields
\[
\sum_{T\in \mathsf T(u)} |T| \lesssim  2^{  u} |\mathsf{sh}(\mathcal T)|
\]
therefore the estimate for the non-local part follows from bounding uniformly in $T\in \mathcal T(u)$
\begin{equation}
\label{e:tails}
\sum_{\substack{ L \in \mathcal L \\ L\subset T}
\\  }
\sum_{t\in \mathbf{T}(L) }  a_t (f\cic{1}_{E\cap U^c}) \lesssim 2^{-10d u}|T|.
\end{equation} To prove the latter estimate, write
\[
\chi(x) \coloneqq \left( 1+ \frac{|\pr_{v_k^\perp} x |^2}{\ell_{I_T}} +   \frac{|\pr_{v_k} x |^2}{\ell_{K_T}}  \right)^{-100d} 
\]
for the rapidly decaying function adapted to the plate $T$. Note that $\phi_t\in \Phi_t^{M}$ with $t\in \mathbf{T}(L)$ is adapted to $R_t$ and thus to the slight enlargement $L$. When $L\subset  T$, $\widetilde{\phi_t}\coloneqq c  \phi_t \chi^{-1}\in \Phi_t^{M/2}$ if $M> 2^{200d} $ and $c>0$ is suitably chosen; in particular we have used that the frequency support of $\widetilde{\phi_t}$ is the same as that of ${\phi_t}$ as $\chi$ is the inverse of a polynomial.

Let $h=f\cic{1}_{E\cap U^c}$ and  $T^{r}= 2^{u+r+1} T\setminus 2^{u+r} T $. As $2^uT\cap U^c=\varnothing,$ we  may write $h= \sum_{r=0}^\infty h \cic{1}_{T^r}.$ Then for suitable choice of $\phi_t\in  \Phi_t^{M}$ and applying \eqref{e:L2est} to pass to the second line
\[
\begin{split}
\sum_{\substack{ L \in \mathcal L \\ L\subset T}} \sum_{t\in \mathbf{T}(L) }  a_t (h \cic{1}_{T^r}) & \leq \sum_{\substack{ L \in \mathcal L \\ L\subset T}}
\sum_{t\in \mathbf{T}(L) }  |\langle h \cic{1}_{T^r},  {\phi_t}\rangle| ^2 
\lesssim
\sum_{\substack{ L \in \mathcal L \\ L\subset T}}
\sum_{t\in \mathbf{T}(L) }  |\langle h \cic{1}_{T^r}\chi,  {\widetilde{\phi}_t}\rangle| ^2 
\\ & \lesssim \|h \cic{1}_{T^r} \chi\|_2^2  \leq  \|\chi \cic{1}_{T^r} \|_\infty^2  |{T^r}| \lesssim  2^{-100d(u+r)} \times (2^{rd} |T|) \lesssim 2^{-99d(u+r)}|T|
\end{split}
\] Summing up over $r\geq 0$ yields the bound \eqref{e:tails}, and completes the proof of the Proposition.
\end{proof}

\bibliographystyle{amsplain}


\end{document}